\documentclass[12pt, draftcls, onecolumn]{IEEEtran}

\IEEEoverridecommandlockouts                              

\usepackage[active]{srcltx}
\usepackage{amsmath}
\usepackage{amssymb}
\usepackage{graphicx}
\usepackage{pstricks}
\usepackage{latexsym}
\usepackage{subfigure}
\usepackage{epsfig}
\usepackage{amsfonts,amssymb}
\usepackage[scanall]{psfrag}
\usepackage[german,english]{babel}
\usepackage{psfrag}
\usepackage{color}
\usepackage{slashbox}
\newtheorem{theorem}            {Theorem}[section]
\newtheorem{definition}         [theorem]{Definition}

\newtheorem{lemma}              [theorem]{Lemma}

\newtheorem{proposition}        [theorem]{Proposition}

\newtheorem{example}        [theorem]{Example}

\newcommand{\dd}   {{\rm d}\hbox{\hskip 0.5pt}}

\newcommand{\calP}{\mathbf{P}}

\newcommand{\rline}{{\mathbb R}}

\newcommand{\bbm}[1]{\left[\begin{matrix} #1 \end{matrix}\right]}
\newcommand{\sbm}[1]{\left[\begin{smallmatrix} #1
   \end{smallmatrix}\right]}

\newcommand{\rfb}[1]{\mbox{\rm
   (\ref{#1})}\ifx\undefined\stillediting\else:\fbox{$#1$}\fi}

\newcommand{\R}{\mathbb{R}}

\newcommand{\bi}{\bibitem}

\newcommand{\bluff}{{\hbox{\raise 15pt \hbox{\hskip 0.5pt}}}}

\newfont{\roma}{cmr10 scaled 1200}

\title{\LARGE \bf Stability analysis and controller design for a linear system with Duhem hysteresis nonlinearity}

\author{ Ruiyue Ouyang, Bayu Jayawardhana\thanks{B. Jayawardhana and Ruiyue Ouyang are with the Dept. Discrete Technology and Production Automation, University of Groningen, Groningen 9747AG, The Netherlands {e-mail: bayujw@ieee.org, r.ouyang@rug.nl}}. }

\begin{document}

\maketitle
\thispagestyle{empty}
\pagestyle{empty}

\begin{abstract}

In this paper, we investigate the stability of a feedback interconnection between a linear system and a Duhem hysteresis operator, where the linear system and the Duhem hysteresis operator satisfy either the counter-clockwise (CCW) or clockwise (CW) input-output dynamics. More precisely, we present sufficient conditions for the stability of the interconnected system that depend on the CW or CCW properties of the linear system and the Duhem operator. Based on these results we introduce a control design methodology for stabilizing a linear plant with a hysteretic actuator or sensor without requiring precise information on the hysteresis operator.
\end{abstract}

\section{Introduction}
Hysteresis is a common phenomenon that is present in diverse systems, such as piezo-actuator, ferromagnetic material and mechanical systems. For describing hysteresis phenomena, several hysteresis models have been proposed in the literature, see, for example, \cite{BROKATE1996, MACKI1993, LOGEMANN2003}. These include backlash model \cite{TARBOURIECH2012} which is used to describe gear trains, Preisach model for modeling the ferromagnetic systems and elastic-plastic model which is used to study mechanical friction \cite{BROKATE1996, MACKI1993}. From the perspective of input-output behavior, the hysteresis phenomena can exhibit counterclockwise (CCW) input-output (I/O) dynamics \cite{ANGELI2006}, clockwise (CW) I/O dynamics \cite{PADTHE2005}, or even more complex I/O map (such as, butterfly map \cite{BERTOTTI2006}). For example, backlash model generates CCW hysteresis loops, elastic-plastic model generates CW hysteresis loops and Preisach model can generate CCW, CW or butterfly hysteresis loops depending on the weight of the hysterons which are used in the Preisach model.
%

The CCW and CW I/O dynamics of a system can also be related to certain dissipation inequalities \cite{ANGELI2006, Petersen2010, SCHAFT2011}. Denoting $AC$ as the class of absolutely continuous functions, we show in \cite{JAYAWARDHANAAUTOMATICA2012} that for a class of Duhem hysteresis operator $\Phi:AC(\rline_+) \times \rline \rightarrow AC(\rline_+)$, we have that for every $u_{\Phi} \in AC(\rline_+)$ there exists a function $H_{\circlearrowleft} : \rline^2 \rightarrow \rline_+$ which satisfies
\begin{equation}\label{passivehysteresis}
\frac{\dd H_{\circlearrowleft}(y_{\Phi}(t),u_{\Phi}(t))}{\dd t} \leq
\dot{y}_{\Phi}(t)u_{\Phi}(t)
\end{equation}
for almost every $t$, where $y_{\Phi}=\Phi(u_{\Phi},y_{\Phi_0})$ and $y_{\Phi_0} \in \rline$ is the initial condition. The inequality \rfb{passivehysteresis} characterizes the CCW I/O property of the operator $\Phi$. We will discuss this property in detail in Section \ref{CCW-CW-Duhem}. Here, we use the symbol $\circlearrowleft$ in $H_{\circlearrowleft}$ to indicate the counterclockwise behavior of $\Phi$.

As a dual result to \cite{JAYAWARDHANAAUTOMATICA2012}, in \cite{RUIYUESCL2012} we give sufficient conditions on the Duhem hysteresis operator such that it exhibits CW input-output dynamics. In particular for a class of Duhem operator $\Phi$, we construct a function $H_{\circlearrowright}: \rline^2 \rightarrow \rline_+$ which satisfies
\begin{equation}\label{passivehysteresis_cw}
\frac{\dd H_{\circlearrowright}(y_{\Phi}(t),u_{\Phi}(t))}{\dd t} \leq
\dot{u}_{\Phi}(t)y_{\Phi}(t)
\end{equation}
for almost every $t$. Correspondingly, the symbol $\circlearrowright$ in $H_{\circlearrowright}$ indicates the clockwise behavior of $\Phi$.

\begin{figure}[h]
\centering \psfrag{P}{$\calP$}\psfrag{p}{${\bf \Phi}$}
 \includegraphics[height=1.2in]{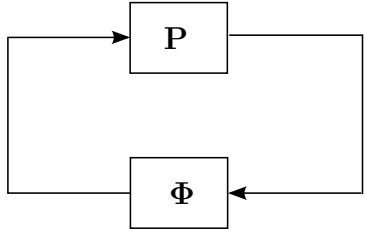}
\caption{Feedback interconnection between a linear plant $\calP$ and a Duhem operator ${\bf \Phi}$.}
\label{blockscheme}
\end{figure}

In this paper, we exploit our knowledge on $H_{\circlearrowleft}$ and $H_{\circlearrowright}$ to study the stability of an interconnected system as shown in Figure \ref{blockscheme}, where $\calP$ is a linear system and ${\bf \Phi}$ is the hysteresis operator. We consider four cases of interconnections where the plant $\calP$ and the hysteresis operator ${\bf \Phi}$ can assume either CCW or CW I/O dynamics. These four cases are summarized in Table \ref{tab1}
\begin{table}[h]
\caption{Four Possible cases of interconnection}
\centering
\begin{tabular}{|l|l|l|}
                                                                                             \hline
                                                                                             \backslashbox{${\bf \Phi}$}{$\calP$} &CCW & CW \\
                                                                                             \hline
                                                                                             CCW & \textcircled{a} & \textcircled{b} \\
                                                                                             CW & \textcircled{c} & \textcircled{d}\\
                                                                                             \hline
\end{tabular}
\end{table}\label{tab1}

In Theorem \ref{SOCCW-CCW} of this paper, the interconnection case $\textcircled{a}$ in Table \ref{tab1} is considered, where both the linear system $\calP$ and the hysteresis operator ${\bf \Phi}$ have CCW I/O dynamics. In particular, we give sufficient conditions on $\calP$ which are dependent on the underlying anhysteresis function of ${\bf \Phi}$ that ensure the stability of the closed-loop system with a positive-feedback interconnection. This result is motivated by recent results on the positive-feedback interconnection of negative imaginary system \cite{Petersen2010} and of CCW systems \cite{SCHAFT2011}. A motivating example for the interconnection case $\textcircled{a}$ is the piezo-actuated stages which are commonly used in the high-precision positioning mechanisms, see, for example \cite{CHIHJER2006}. The piezo-actuated stage contains two parts: a piezo-actuator and a positioning mechanism, which can be described by
\begin{equation}\label{Piezosystem}
\left.\begin{array}{rr}
\calP :&\begin{array}{rl} m\ddot{x}+b\dot{x}+kx & = F_{{\rm piezo}},\\
                            V&=cx,\end{array}\\[0.5cm]
{\bf \Phi}:&  F_{{\rm piezo}} =\Phi(V),\hspace{0.25cm}\end{array}\right\}
\end{equation}
where $m$ is the mass, $b$ is the damping constant, $k$ is the spring constant, $c$ is the proportional gain, $V$ is the input voltage of the piezo-actuator, $F_{{\rm piezo}}$ denotes the force generated by the piezo-actuator and $x$ denotes the displacement of the stage. The piezoelectric actuator has been shown to have CCW hysteresis loops from the input voltage to the output generated force (see, for example \cite{DRAGAN2005}). It can be checked that the linear mass-damper-spring system $\calP$ is also CCW from $F_{{\rm piezo}}$ to $x$ or, equivalently, $\calP$ is a negative-imaginary system \cite{Petersen2010}.

In Theorem \ref{SOCW-CCW}, we consider the interconnection case $\textcircled{b}$ in Table \ref{tab1}, where the linear system $\calP$ has CW I/O dynamics and the hysteresis operator ${\bf \Phi}$ has CCW I/O dynamics. In this case Theorem \ref{SOCW-CCW} provides sufficient conditions on $\calP$ which are independent of ${\bf \Phi}$ such that the closed-loop system with a negative feedback interconnection is stable. An example for this case is the active vibration mechanism using piezo-actuator, which has been used for vibration control in mechanical structures \cite{KAMADA1997}. The mechanism can be described by
\begin{equation}\label{RLCsystem}
\left.\begin{array}{rr}
\calP :&\begin{array}{rl} m\ddot{x}+b\dot{x}+kx & = F_{{\rm piezo}},\\
                            V&=-c\ddot{x},\end{array}\\ [0.5cm]
{\bf \Phi}:& F_{{\rm piezo}} =\Phi(V). \hspace{0.25cm} \end{array}\right\}
\end{equation}
As described before, the piezoelectric actuator has CCW I/O dynamics and it can be checked that the mass-damper-spring system $\calP$ is CW from $F_{{\rm piezo}}$ to $\ddot{x}$.

Theorem \ref{SOCCW-CW} deals with the interconnection case $\textcircled{c}$, where $\calP$ has CCW I/O dynamics and ${\bf \Phi}$ has CW I/O dynamics. A motivating example for this case is the mechanical systems with friction \cite{PADTHE2008}, which is given by
\begin{equation}\label{Fricsystem}
\left.\begin{array}{rrl}
\calP :& m\ddot{x}+kx & = - F_{{ \rm friction}},\\[0.5cm]
{\bf \Phi}:& F_{{ \rm friction}}& =\Phi(x),\end{array}\right\}
\end{equation}
where $F_{{ \rm friction}}$ is the friction force. As discussed in \cite{PADTHE2008}, the friction force has CW I/O dynamics where the input is the displacement. On the other hand, the mechanical system is CCW from the friction force $-F_{{ \rm friction}}$ to the displacement $x$.

As a completion to the Table \ref{tab1}, we present the analysis of the interconnection case $\textcircled{d}$ in Theorem \ref{CW-CW}. Based on these results, we present in Section \ref{control-design} a control design methodology for a linear plant with a hysteretic actuator/sensor ${\bf \Phi}$ and we provide two numerical examples in Section VII.


\section{preliminaries} \label{CCW-CW-Duhem}
In this section we give the definitions of the CCW and CW dynamics based on the work by Angeli \cite{ANGELI2006} and Padthe \cite{PADTHE2005}. Figure \ref{ccw-cw} illustrates the CCW and CW input-output dynamics of a (nonlinear) operator $G: u \mapsto G(u)=:y$. We denote $AC(\rline_+, \rline^n)$ the space of absolutely continuous function $f:\rline_+ \rightarrow \rline^n$.
\begin{figure}[h]
\centering \psfrag{u}{$u$}\psfrag{y}{$y$}
  \subfigure[]{\includegraphics[width=0.2\textwidth]{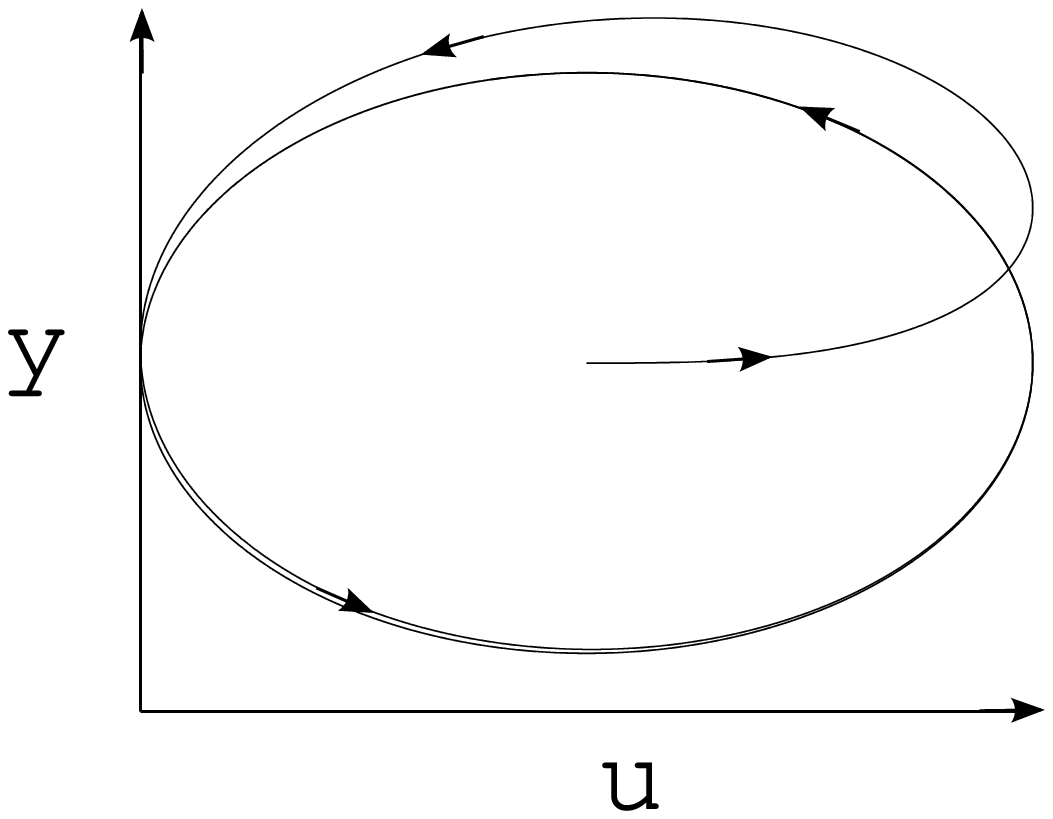}}
  \subfigure[]{\includegraphics[width=0.2\textwidth]{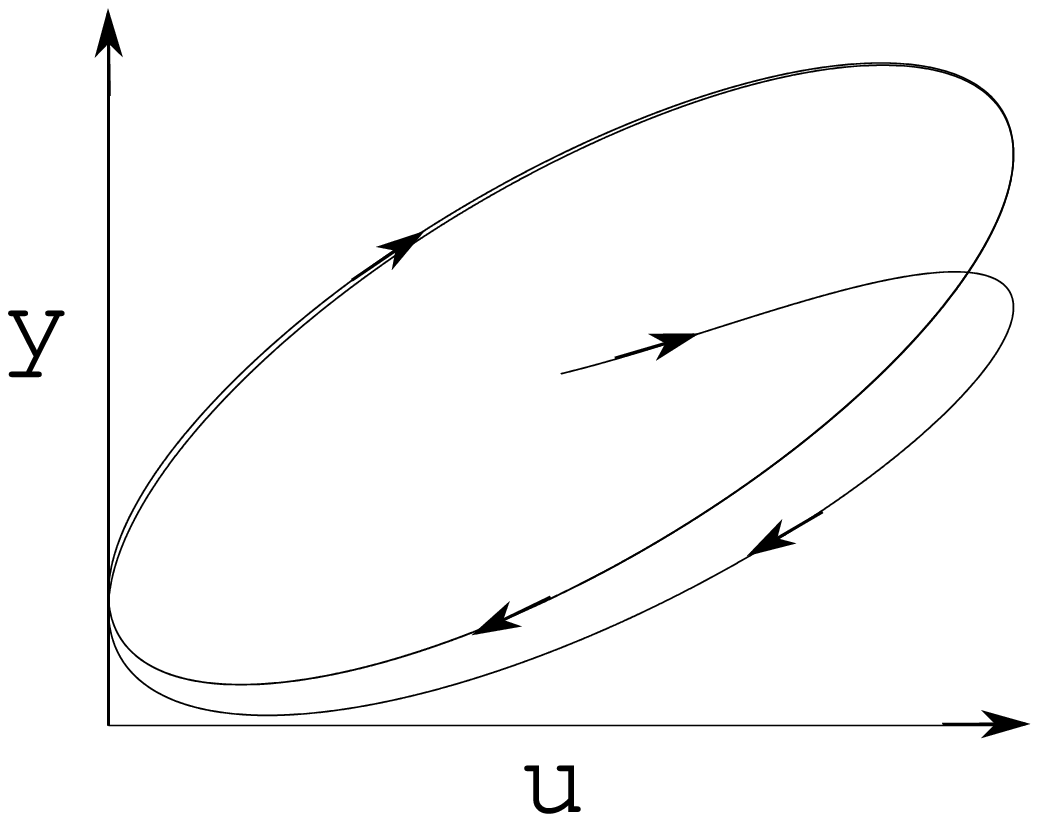}}
\caption{A graphical illustration of counter-clockwise (CCW) and clockwise (CW) I/O dynamics of an operator $G:u \longmapsto y$. $(a)$ CCW I/O dynamics; $(b)$ CW I/O dynamics.}
\label{ccw-cw} \vspace{0.2cm}
\end{figure}

\subsection{Counterclockwise dynamics}
\begin{definition} \cite{ANGELI2006, PADTHE2005}
 A (nonlinear) map $G:AC(\rline_+, \rline^m)\rightarrow AC(\rline_+, \rline^m)$ is {\it counterclockwise} (CCW) if for every $u \in AC(\rline_+, \rline^m)$ with the corresponding output map $y:=Gu$, the following inequality holds
\begin{equation}\label{deCCW}
\liminf_{T\rightarrow \infty}\int^T_0{\langle \dot{y}(t), u(t)\rangle dt > -\infty}.
\end{equation}
\end{definition}
\vspace{0.2cm}
For an operator $G$, inequality \rfb{deCCW} holds if there exists a function $V: \rline^2 \rightarrow \rline_+$ such that for every input signal $u$, the inequality
\begin{equation}
\frac{\dd V(y(t),u(t))}{\dd t} \leq \langle
\dot{y}(t), u(t)\rangle,
\end{equation}
holds for almost every $t$ where the output signal $y := Gu$.

\begin{definition}\label{def_soccw}
A (nonlinear) map $G:AC(\rline_+, \rline^m)\rightarrow AC(\rline_+, \rline^m)$ is {\it strictly counterclockwise} (S-CCW) (see also \cite{ANGELI2006}), if for every input $u \in AC(\rline_+, \rline^m)$, there exists a constant $\varepsilon> 0$ such that the inequality
 \begin{equation}
\liminf_{T\rightarrow \infty}\int^T_0{\langle \dot{y}(t),\ u(t)\rangle -\varepsilon \|\dot{y}(t)\|^2\dd t > -\infty},
\end{equation}
holds where $y := Gu$.
\end{definition}
Note that for systems described by the state space representation as follows:
\begin{equation}\label{statespace}
\Sigma: \left. \begin{array}{rl}
         \dot{x}&=f(x,u), \qquad x(0)=x_0\\
          y&=h(x),
        \end{array} \right\}
\end{equation}
where $x(t) \in \rline^n$ is the state, $u(t) \in \rline^m$ is the input, $y(t) \in \rline^m$ is the output and $f$, $h$ are sufficiently smooth functions, the following lemma provides sufficient conditions for $\Sigma$ to be CCW (and S-CCW).
\begin{lemma}\label{lemma1}
Consider the state space system $\Sigma$ as in \rfb{statespace}. If there exists $V :\rline^n \rightarrow \rline_+$ and $\varepsilon \geq 0$, such that
\begin{equation*}
\frac{\partial V(x)}{\partial x}f(x,u) \leq \left\langle \frac{\partial h(x)}{\partial x}f(x,u), u\right\rangle-\varepsilon \left\| \frac{\partial h(x)}{\partial x}f(x,u)\right\|^2,
\end{equation*}
holds for all $x\in\rline^n$ and $u\in\rline^m$, then $\Sigma$ is CCW. Moreover if $\varepsilon > 0$, it is S-CCW.
\end{lemma}

\subsection{Clockwise dynamics}
Dual to the concept of counterclockwise I/O dynamics, the notion of clockwise I/O dynamics can be defined as follows.

\begin{definition}\cite{PADTHE2005}\label{de-CW}
A (nonlinear) map $G:AC(\rline_+, \rline^m)\rightarrow AC(\rline_+, \rline^m)$ is {\it clockwise} (CW) if for every input $u \in AC(\rline_+, \rline^m)$ with the corresponding output map $y:=Gu$, the following inequality holds:
\begin{equation}\label{defCW}
\liminf_{T\rightarrow \infty}\int^T_0{y(t)^T\dot{u}(t)\dd t > -\infty}.
\end{equation}
\end{definition}

\vspace{0.2cm}
For a nonlinear operator $G$, inequality \rfb{defCW} holds if there exists a function $V: \rline^2 \rightarrow \rline_+$ such that for every input signal $u\in AC(\rline_+, \rline^m)$, the inequality
\begin{equation}
\frac{\dd V(y(t),u(t))}{\dd t} \leq \langle
y(t), \dot{u}(t)\rangle,
\end{equation}
holds for a.e. $t$ where the output signal $y := Gu$.
\begin{lemma}
Consider the state space system $\Sigma$ as in \rfb{statespace}. If there exist $\alpha, V:\rline^{m+n} \rightarrow \rline_+$, such that $V$ is positive definite and proper, and
\begin{equation}\label{lyaCW}
\left[\begin{array}{cc}
  \frac{\partial V(w,x)}{\partial w} & \frac{\partial V(w,x)}{\partial x}
\end{array}\right]\left[ \begin{array}{c}
               q \\
               f(x,w)
             \end{array}\right] \leq \langle h(x), w\rangle - \alpha(w,x) ,
\end{equation}
holds for all $x\in\rline^n$, $w\in\rline^m$ and $q\in\rline^m$, then $\Sigma$ is CW.
\end{lemma}
\begin{proof}
Define the extended state space system \rfb{statespace} as follows
\begin{equation}\label{exstatespace}
 \left.\begin{array}{rl}
         \dot{w} &=q,\\
         \dot{x}&=f(x,w), \\
          y&=h(x).
        \end{array}\right.
\end{equation}
Note that $w$ defines the input in \rfb{statespace}. It follows from \rfb{lyaCW} and \rfb{exstatespace} that
\begin{align*}
\dot{V} &\leq \langle h(x), q\rangle - \alpha(x,w), \\
        &= \langle y, \dot{w}\rangle - \alpha(x,w),
\end{align*}
which completes our proof by taking $w=u$.
\end{proof}
\vspace{0.2cm}
\section{Duhem Hysteresis operator}\label{Duhem}
The Duhem operator $\Phi:AC(\rline_+)\times \R\to AC(\rline_+), (u_{\Phi},y_{\Phi_0})\mapsto \Phi(u_{\Phi},y_{\Phi_0})=:y_{\Phi}$ is described by
\begin{equation}\label{babuskamodel}
\dot y_{\Phi}(t) = f_1(y_{\Phi}(t),u_{\Phi}(t))\dot u_{\Phi+}(t) + f_2(y_{\Phi}(t),u_{\Phi}(t))\dot u_{\Phi-}(t), \ y_{\Phi}(0)=y_{\Phi_0},
\end{equation}
where $\dot u_{\Phi+}(t):=\max\{0,\dot u_{\Phi}(t)\}$, $\dot u_{\Phi-}(t):=\min\{0,\dot u_{\Phi}(t)\}$ and
$f_1:\R^2\to\R$, $f_2:\R^2\to\R$ are $C^1$. We refer to \cite{MACKI1993, OH2005, VISINTIN1994} for standard properties of the Duhem operator, such as causality, monotonicity and rate-independency.

The existence of solutions to \rfb{babuskamodel} has been reviewed in \cite{MACKI1993}. In particular, if for every $v \in \rline $, the functions $f_1$ and $f_2$ satisfy
\begin{align}\label{exis_cond}
(\gamma_1 - \gamma_2)[f_1(\gamma_1,v) - f_1(\gamma_2,v)] & \leq \lambda_1(v)(\gamma_1 - \gamma_2)^2,\\\nonumber
(\gamma_1 - \gamma_2)[f_2(\gamma_1,v) - f_2(\gamma_2,v)] & \geq -\lambda_2(v)(\gamma_1 - \gamma_2)^2,
\end{align}
for all $\gamma_1$, $\gamma_2 \in \rline$, where $\lambda_1$ and $\lambda_2$ are nonnegative, then \rfb{babuskamodel} has a unique global solution and $\Phi$ maps $AC(\rline_+) \times \rline \rightarrow AC(\rline_+)$.

\subsection{Duhem operator with CCW characterization}\label{DuhemCCW}
To show the CCW properties of the Duhem operator, we review our previous results in \cite{JAYAWARDHANAAUTOMATICA2012}. In \cite{JAYAWARDHANAAUTOMATICA2012}, we define a function $H_{\circlearrowleft}:\R^2 \to \rline_+$ for the Duhem operator $\Phi$ such that \rfb{passivehysteresis} holds (under certain conditions on $f_1$ and $f_2$).
Before we can define the function $H_{\circlearrowleft}$ for $\Phi$, we need to define three functions which depend on $f_1$ and $f_2$.

Firstly, we define a traversing function $\omega_{\Phi}$ which describes the possible trajectory of $\Phi$ when a monotone increasing $u_{\Phi}$ and a monotone decreasing $u_{\Phi}$ is applied to $\Phi$ from an initial condition.

For every pair $(y_{\Phi_0},u_{\Phi_0})\in \rline^2$, let
$\omega_{\Phi,1}(\cdot,y_{\Phi_0},u_{\Phi_0}):[u_{\Phi_0},\infty)\to\rline$ be the
solution of
\begin{equation*}
z(v) - y_{\Phi_0} = \int^v_{u_{\Phi_0}}{f_1(z(\sigma),\sigma)\
\dd\sigma}, \quad \forall v\in[u_{\Phi_0},\infty),
\end{equation*}
and let $\omega_{\Phi,2}(\cdot,y_{\Phi_0},u_{\Phi_0}):(-\infty,u_{\Phi_0}]\to\rline$ be
the solution of
\begin{equation*}
z(v) - y_{\Phi_0} = \int_{u_{\Phi_0}}^v{f_2(z(\sigma),\sigma)\
\dd\sigma},\quad \forall v\in (-\infty,u_{\Phi_0}].
\end{equation*}
Using the above definitions, for every pair $(y_{\Phi_0},u_{\Phi_0})\in \rline^2$, the {\it traversing function}
$\omega_\Phi(\cdot,y_{\Phi_0},u_{\Phi_0}):\rline\to\rline$ is defined by the
concatenation of $\omega_{\Phi,2}(\cdot,y_{\Phi_0},u_{\Phi_0})$ and
$\omega_{\Phi,1}(\cdot,y_{\Phi_0},u_{\Phi_0})$:
\begin{equation}\label{wcurve}
\omega_\Phi(v,y_{\Phi_0},u_{\Phi_0}) = \left\{ \begin{array}{ll}
\omega_{\Phi,2}(v,y_{\Phi_0},u_{\Phi_0}) & \forall v \in (-\infty,u_{\Phi_0}) \\
\omega_{\Phi,1}(v,y_{\Phi_0},u_{\Phi_0}) & \forall v \in [u_{\Phi_0},\infty).
\end{array} \right.
\end{equation}
Again, we remark that the curve $\omega_\Phi(\cdot,y_{\Phi_0},u_{\Phi_0})$ is the (unique) hysteresis
curve where the curve defined in $(-\infty, u_{\Phi_0}]$ is obtained by
applying a monotone decreasing $u_{\Phi}\in AC(\rline_+, \rline^m)$ to $\Phi(u_{\Phi},y_{\Phi_0})$ with
$u_{\Phi}(0)=u_{\Phi_0}$ and $\lim_{t\to\infty}u_{\Phi}(t)=-\infty$ and,
similarly, the curve defined in $[u_{\Phi_0},\infty)$ is produced by
introducing a monotone increasing $u_{\Phi}\in AC(\rline_+, \rline^m)$ to $\Phi(u_{\Phi},y_{\Phi_0})$
with $u_{\Phi}(0)=u_{\Phi_0}$ and $\lim_{t\to\infty}u_{\Phi}(t)=\infty$.

The second function we need to define is the {\it anhysteresis function} $f_{an}$, which represents the curve where $f_1(f_{an}(v),v)=f_2(f_{an}(v),v)$.

Another function that is needed for defining $H_{\circlearrowleft}$ is the intersecting function between the anhysteresis function $f_{an}$ and the function $\omega_{\Phi}$ as defined above. The function $\Omega:\rline^2\rightarrow \rline$ is the {\it CCW intersecting function} if $\omega_\Phi(\Omega(\gamma,v),\gamma,v)=f_{an}(\Omega(\gamma,v))$ for all $(\gamma, v)\in\rline^2$ and $\Omega(\gamma,v)\geq v$ whenever $\gamma \geq f_{an}(v)$ and $\Omega(\gamma,v)<v$ otherwise. For simplicity, we assume that $\Omega$ is differentiable. In \cite[Lemma 3.1]{JAYAWARDHANAAUTOMATICA2012} sufficient conditions on $f_1$ and $f_2$ which guarantee the existence of such $\Omega$ are $f_{an}$ be monotone increasing and
\begin{eqnarray}
\label{eq_ass_Omega_11}
f_1(\gamma,v) <  \frac{\dd f_{an}(v)}{\dd v}-\epsilon & \textrm{ whenever }& \gamma > f_{an}(v)\ \\
\label{eq_ass_Omega_12}
f_2(\gamma,v) <  \frac{\dd f_{an}(v)}{\dd v}-\epsilon & \textrm{ whenever }& \gamma < f_{an}(v)\
\end{eqnarray}
hold with $\epsilon>0$.

\begin{theorem}\label{CCW_hys}
Consider the Duhem hysteresis operator $\Phi$ defined in \rfb{babuskamodel} with $C^1$ functions $f_1, f_2:\rline^2 \to \rline_+$. Let $f_{an}$ be the corresponding anhysteresis function which is monotone increasing and satisfies \rfb{eq_ass_Omega_11} and \rfb{eq_ass_Omega_12}. Denote by $\Omega$ the corresponding CCW intersecting function. Suppose that for all $(\gamma,v)$ in $\rline^2$, $f_1(\gamma,v)\geq f_2(\gamma,v)$ whenever $\gamma \leq f_{an}(v)$ and $f_1(\gamma,v) < f_2(\gamma,v)$ otherwise. Then $\Phi$ is CCW with the function $H_{\circlearrowleft}:\rline^2\rightarrow \rline_+$ be given by
\begin{equation}\label{storage}
H_{\circlearrowleft}(\gamma,v) = \gamma v -\int_0^v{\omega_\Phi(\sigma,\gamma,v) \ \dd\sigma} + \int_0^{\Omega(\gamma,v)}{\omega_\Phi(\sigma,\gamma,v) - f_{an}(\sigma) \ \dd\sigma}.
\end{equation}
\end{theorem}

\begin{proof}
The proof follows from Lemma 3.1 and Theorem 3.3 in \cite{JAYAWARDHANAAUTOMATICA2012}. In particular, it is shown in \cite{JAYAWARDHANAAUTOMATICA2012} that
\begin{equation}\label{storCCW}
\frac{\dd H_{\circlearrowleft}(y_{\Phi}(t),u_{\Phi}(t))}{\dd t} \leq \langle \dot{y}_{\Phi}(t), u_{\Phi}(t)\rangle,
\end{equation}
where $y_{\Phi}:=\Phi(u_{\Phi},y_{\Phi_0})$ and $H_{\circlearrowleft}$ is non-negative.
By integrating \rfb{storCCW} from $0$ to $T$ we have
\begin{equation*}
H_{\circlearrowleft}\big(y_{\Phi}(T), u_{\Phi}(T)\big) - H_{\circlearrowleft}\big(y_{\Phi}(0), u_{\Phi}(0)\big)\\ = \int_0^T{\dot
y_{\Phi}(\tau) u_{\Phi}(\tau)\dd \tau}.
\end{equation*}
Since $H_{\circlearrowleft}$ is nonnegative then
\begin{equation*}
\int_0^T{\dot y_{\Phi}(\tau) u_{\Phi}(\tau)\dd \tau}  \geq  - H_{\circlearrowleft}(y_{\Phi}(0),u_{\Phi}(0))>-\infty.
\end{equation*}
\end{proof}

An example of the CCW hysteresis phenomenon is the magnetic hysteresis in ferromagnetic material, which has CCW behavior from the input (an applied electrical field) to the output (the magnetization). The magnetic hysteresis can be modeled by the Coleman-Hodgdon model \cite{COLEMAN1986} given by
\begin{equation}\label{coleman}
   \dot{y}_{\Phi}(t) = C_{\alpha} |\dot{u}_{\Phi}(t)|[f(u_{\Phi}(t))-y_{\Phi}(t)]+\dot{u}_{\Phi}(t)g(u_{\Phi}(t)),
\end{equation}
where $C_{\alpha}$ is a positive constant, $f: \rline \rightarrow \rline$ is a monotone
increasing $C^1$ function, such that $f(0)=0$ and $g$ is locally Lipschitz.
The Coleman-Hodgdon model in \rfb{coleman} can be rewritten into the form of \rfb{babuskamodel} where:
\begin{equation}\label{coleman_duhem}
 f_1(y_{\Phi},u_{\Phi})= C_{\alpha} [f(u_{\Phi})-y_{\Phi}]+g(u_{\Phi}), \ f_2(y_{\Phi},u_{\Phi})=-C_{\alpha} [f(u_{\Phi})-y_{\Phi}]+g(u_{\Phi})
\end{equation}
In this case, it has the same structure as in \rfb{babuskamodel} with $f_{an}=f$. Figure \ref{colemanhys} shows the behaviour of the Coleman-Hodgdon model using the functions $f$ and $g$ given by
\begin{equation}\label{coleman_ex}
f(u_{\Phi}) = bu_{\Phi}, \quad g(u_{\Phi})= a,
\end{equation}
where $b>0$ and $a>0$. It can be easily checked that for every $u_{\Phi}(t) \in \rline$, $f_1$ and $f_2$ satisfy \rfb{exis_cond}, i.e., for every $u_{\Phi}\in AC(\rline_+)$ and for every $y_{\Phi}(0) \in \rline$, the solution of \rfb{coleman_duhem} exists for all $t \in \rline_+$.
\begin{figure}[h]
\centering \psfrag{y}{$y(t)$}\psfrag{u}{$u(t)$}\psfrag{t}{$t$} \psfrag{yp}{$y_{\Phi}$}\psfrag{up}{$u_{\Phi}$}\psfrag{omg}{$\Omega(y_{\Phi},u_{\Phi})$} \psfrag{fan}{$f_{an}$}
  \subfigure[]{\includegraphics[width=2.4in]{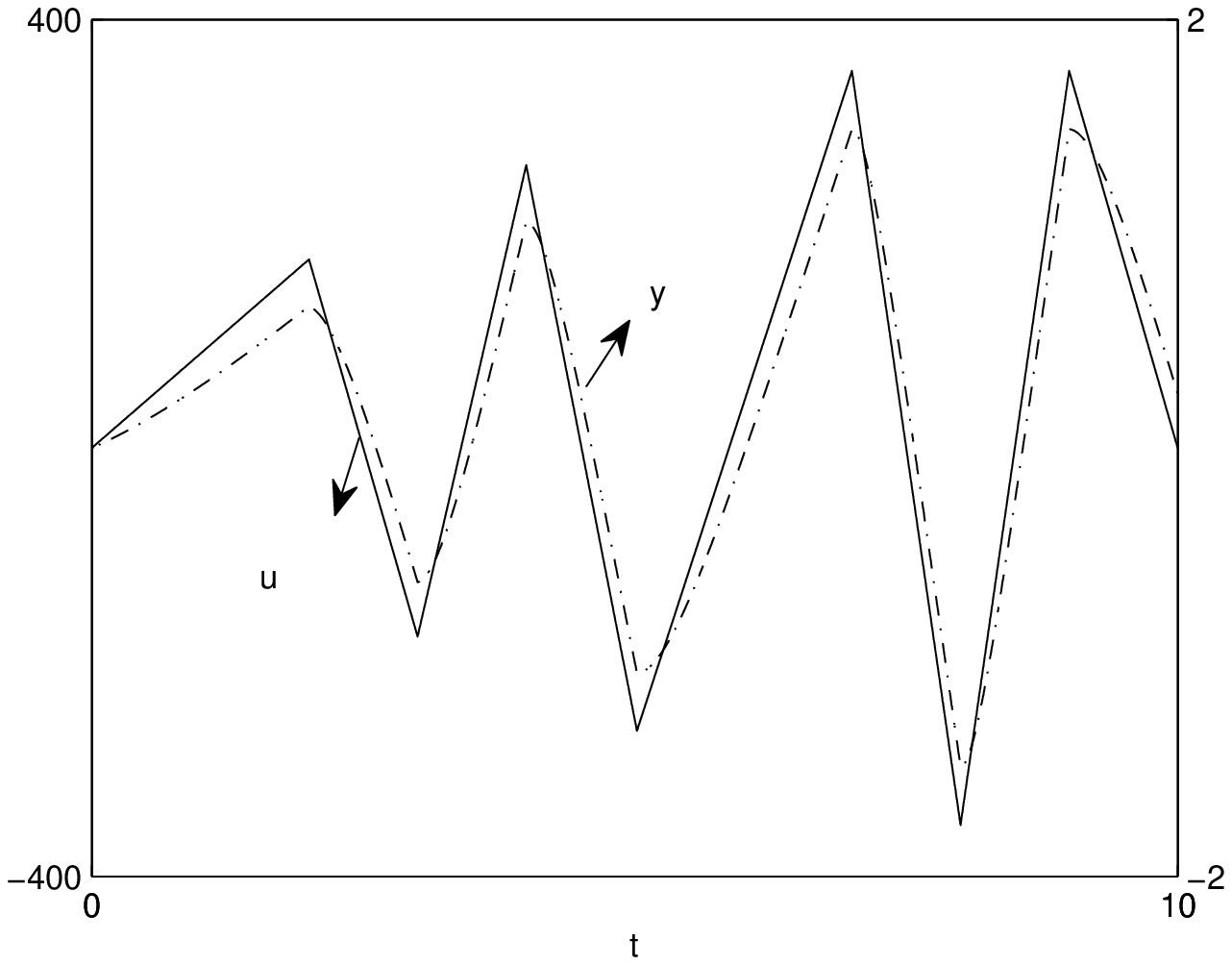}}
  \subfigure[]{\includegraphics[width=2.4in]{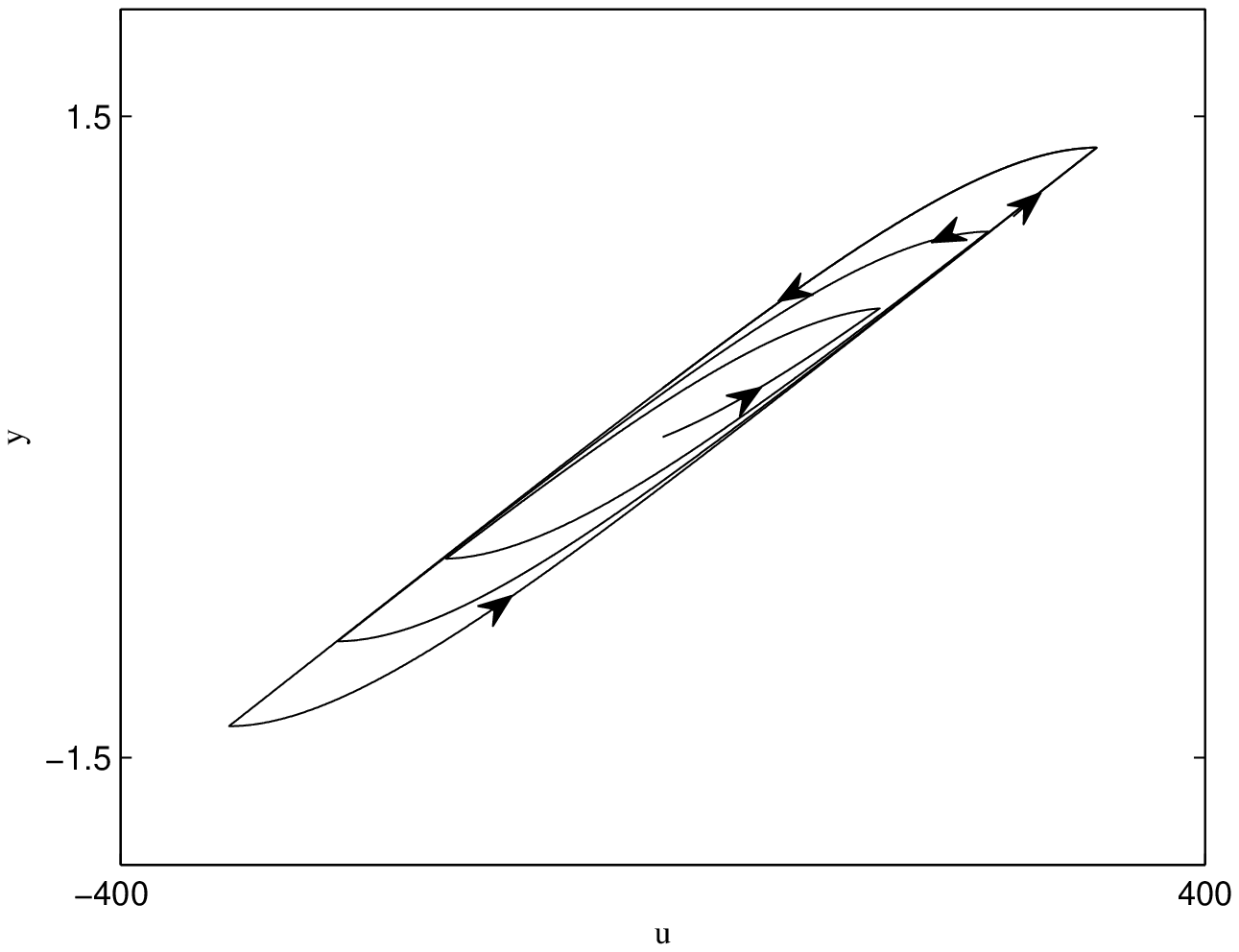}}
\caption{Behaviour of the Coleman-Hodgdon model using $f$ and $g$ as in \rfb{coleman_ex} with $b=5 \times 10^{-3}$, $C_{\alpha}=1 \times 10^{-2}$, $a=2.5 \times 10^{-3}$ and $y_{\Phi_0}=0$.}
\label{colemanhys} \vspace{0.2cm}
\end{figure}

Calculating the curve $\omega_{\Phi}$, we have
\begin{equation}
\omega_{\Phi}(\sigma,y_{\Phi}(t),u_{\Phi}(t))=\left \{ \begin{array}{ll} b\sigma+\frac{a-b}{C_{\alpha}}+(y_{\Phi}(t)-bu_{\Phi}(t)+\frac{b-a}{C_{\alpha}})\mathop{e}^{-C_{\alpha}(\sigma-u_{\Phi})}\ \ \sigma\in[u_{\Phi}(t),\ \infty),\\
b\sigma+\frac{b-a}{C_{\alpha}}+(y_{\Phi}(t)-bu_{\Phi}(t)+\frac{a-b}{C_{\alpha}})\mathop{e}^{C_{\alpha}(\sigma-u_{\Phi})} \ \ \sigma \in(-\infty,\ u_{\Phi}(t)]. \end{array} \right.
\end{equation}
The CCW intersecting function $\Omega(y_{\Phi}(t),u_{\Phi}(t))$ is given by
\begin{equation}\label{intersectionfun}
\Omega(y_{\Phi}(t),u_{\Phi}(t))=\left \{ \begin{array}{ll} u_{\Phi}(t)- \frac{1}{C_{\alpha}}{\rm ln} \left [\frac{\frac{b-a}{C_{\alpha}}}{y_{\Phi}(t)-bu_{\Phi}(t)+\frac{b-a}{C_{\alpha}}}\right ]\ y_{\Phi}(t) \geq f_{an}(u_{\Phi}(t)),\\
u_{\Phi}(t)+ \frac{1}{C_{\alpha}}{\rm ln}  \left [\frac{\frac{a-b}{C_{\alpha}}}{y_{\Phi}(t)-bu_{\Phi}(t)+\frac{a-b}{C_{\alpha}}}\right ] \ y_{\Phi}(t) < f_{an}(u_{\Phi}(t)). \end{array} \right.
\end{equation}
Since $f_1$ and $f_2$ satisfy the hypotheses in Theorem \ref{CCW_hys}, $\Phi$ is CCW. Denoting $u_{\Phi}^*(t)=\Omega(y_{\Phi}(t),u_{\Phi}(t))$, we can compute explicitly $H_{\circlearrowleft}$ in \rfb{storage} as follows

%
\begin{equation}\label{storageDahl}
H_{\circlearrowleft}(y_{\Phi}(t),u_{\Phi}(t))\\=\left \{ \begin{array}{ll} u_{\Phi}(t)y_{\Phi}(t)-\frac{1}{2}bu_{\Phi}(t)^2+\frac{a-b}{C_{\alpha}}(u_{\Phi}^*(t)-u_{\Phi}(t))\\
+\frac{1}{C_{\alpha}}(y_{\Phi}(t)-bu_{\Phi}(t)+\frac{b-a}{C_{\alpha}})(1-\mathop{e} ^{C_{\alpha}(u_{\Phi}(t)-u_{\Phi}^*(t))}) \quad y_{\Phi}(t) \geq f_{an}(u_{\Phi}(t)),\\
u_{\Phi}(t)y_{\Phi}(t)-\frac{1}{2}bu_{\Phi}(t)^2+\frac{b-a}{C_{\alpha}}(u_{\Phi}^*(t)-u_{\Phi}(t))\\
+\frac{1}{C_{\alpha}}(y_{\Phi}(t)-bu_{\Phi}(t)+\frac{a-b}{C_{\alpha}})(\mathop{e}^{C_{\alpha}(u_{\Phi}^*(t)-u_{\Phi}(t))}-1) \quad y_{\Phi}(t) \leq f_{an}(u_{\Phi}(t)). \end{array} \right.
\end{equation}

The graphical interpretation of $H_{\circlearrowleft}$ is shown in Figure \ref{colemanhys_stor}, where the value of $H_{\circlearrowleft}$ at a given time $t$ is given by the area in grey.
\begin{figure}[h]
\centering \psfrag{H}{$H_{\circlearrowleft}(y_{\Phi}(t),u_{\Phi}(t))$}\psfrag{u}{$u(t)$}\psfrag{t}{$t$} \psfrag{yp}{$y_{\Phi}(t)$}\psfrag{up}{$u_{\Phi}(t)$}\psfrag{omg}{$\Omega(y_{\Phi}(t),u_{\Phi}(t))$} \psfrag{fan}{$f_{an}$} \includegraphics[width=2.8in]{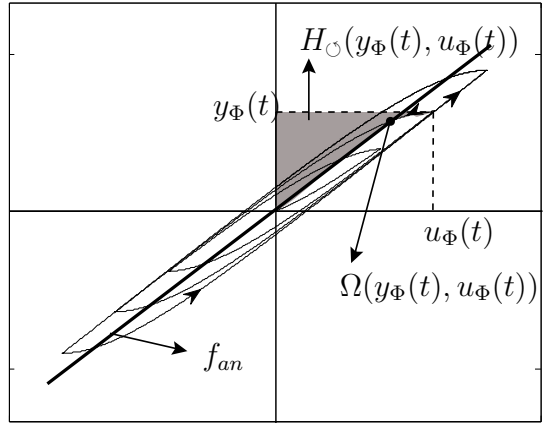}
\caption{Graphical interpretation of the function $H_{\circlearrowleft}(y_{\Phi}(t),u_{\Phi}(t))$ of the Coleman-Hodgdon model using $f$ and $g$ as in \rfb{coleman_ex} with $b=5 \times 10^{-3}$, $C_{\alpha}=1 \times 10^{-2}$, $a=2.5 \times 10^{-3}$ and $y_{\Phi_0}=0$.}
\label{colemanhys_stor}
\end{figure} \vspace{0.2cm}

\begin{proposition}
Consider the Duhem operator $\Phi$ satisfying the hypotheses in Theorem \ref{CCW_hys}. Suppose that $f_{an}(0)=0$. Then the function $H_{\circlearrowleft}(\cdot,v)$ (where $H_{\circlearrowleft}$ is as in \rfb{storage}) is radially unbounded for every $v$.
\end{proposition}
\begin{proof}
Let us consider $v>0$. To show the properness of $H_{\circlearrowleft}(\cdot,v)$, let us first consider the case where $\gamma \geq f_{an}(v)$. In this case, we rewrite the function $H_{\circlearrowleft}$, as follows
\begin{equation*}
H_{\circlearrowleft}(\gamma,v)=\int_{0}^{v}{\gamma -f_{an}(\sigma) \dd \sigma}+\int_{v}^{\Omega(\gamma, v)}{\omega_{\Phi}(\sigma,\gamma,v)-f_{an}(\sigma) \dd \sigma}
\end{equation*}
Due to the property of the CCW intersecting function $\Omega$, $\gamma \geq f_{an}(v)$ implies that $\Omega(\gamma,v)\geq v$. Hence the last term on the RHS of the above equation is non-negative, i.e., $\int_{v}^{\Omega(\gamma, v)}{\omega_{\Phi}(\sigma,\gamma,v)-f_{an}(\sigma) \dd \sigma} \geq 0$. Then,
\begin{equation}\label{rad_ccw}
H_{\circlearrowleft}(\gamma,v)\geq \int_{0}^{v}{\gamma -f_{an}(\sigma) \dd \sigma} \geq \int_{0}^{v}{\gamma -f_{an}(v) \dd \sigma}=(\gamma-c)v,
\end{equation}
where $c:=f_{an}(v)$. Equation \rfb{rad_ccw} indicates that for every $v>0$, $H_{\circlearrowleft}(\gamma, v)\rightarrow \infty$ as $\gamma \rightarrow \infty$.
%

To evaluate the other limit when $\gamma \rightarrow -\infty$, let us consider the case when $\gamma < 0$. Note that in this case $\gamma < f_{an}(v)$ due to the monotonicity assumption on $f_{an}$ and $f_{an}(0)=0$. Rewriting $H_{\circlearrowleft}$, we have
%
\begin{align*}
H_{\circlearrowleft}(\gamma,v)&=\int_{0}^{\Omega(\gamma, v)}{\gamma  -f_{an}(\sigma)\dd \sigma}+ \int_{\Omega(\gamma, v)}^{v}{\gamma - \omega_{\Phi}(\sigma,\gamma,v)\dd \sigma} \\
& \geq \int_{0}^{\Omega(\gamma, v)}{\gamma  -f_{an}(\sigma)\dd \sigma} =\int^{0}_{\Omega(\gamma, v)}{f_{an}(\sigma)- \gamma \dd \sigma}.
\end{align*}
The last inequality is obtained due to the property of the CCW intersecting function $\Omega$, where $\Omega(\gamma,v) < v$ whenever $\gamma < f_{an}(v)$. Since $\omega_{\Phi}$ is monotone and non-decreasing (due to the positivity of $f_1$ and $f_2$) and using the fact that $f_{an}$ is monotone increasing and $f_{an}(0)=0$, it can be checked that $\gamma <0$ implies that $\Omega(\gamma,v) <0$.

Now let us fix $\bar{\gamma}$ such that $0>\bar{\gamma}>\gamma$. Using the fact that $\omega_{\Phi}(\sigma,\bar{\gamma},v) \geq \omega_{\Phi}(\sigma,\gamma,v)$ for all $\sigma < v$ and using monotonicity of $f_{an}$, it follows that $0>\bar{\Omega}>\Omega(\gamma,v)$ where the constant $\bar{\Omega}:=\Omega(\bar{\gamma},v)$. Thus
 \begin{align*}
H_{\circlearrowleft}(\gamma,v) &\geq \int^{0}_{\Omega(\gamma, v)}{f_{an}(\sigma)- \gamma \dd \sigma}\\&= \int_{\Omega(\gamma, v)}^{\bar{\Omega}}{f_{an}(\sigma)- \gamma \dd \sigma}+\int^{0}_{\bar{\Omega}}{f_{an}(\sigma)- \gamma \dd \sigma} \\
& \geq \int^{0}_{\bar{\Omega}}{f_{an}(\sigma)- \gamma \dd \sigma} \geq \int^{0}_{\bar{\Omega}}{f_{an}(\bar{\Omega})- \gamma \dd \sigma}\\
&=(\gamma-f_{an}(\bar{\Omega}))\bar{\Omega}.
\end{align*}
The last equality shows that as $\gamma \rightarrow - \infty$, $H_{\circlearrowleft}\rightarrow \infty$ since $\bar{\Omega}<0$. Therefore, we can conclude that for the case $v >0$, the function $H_{\circlearrowleft}(\cdot,v)$ is radially unbounded.

Using similar arguments we can get the same conclusion for the case when $v \leq 0$.
\end{proof}
\subsection{Duhem operator with CW characterization}\label{DuhemCW}
The CW property of the Duhem operator has been discussed in our previous results in \cite{RUIYUESCL2012}, where we also constructed a function $H_{\circlearrowright}: \rline^2 \rightarrow \rline_+$ for the Duhem operator such that \rfb{passivehysteresis_cw} holds. Following a similar procedure as before, the construction of the function $H_{\circlearrowright}$ requires three functions: the traversing function $\omega_{\Phi}$, the anhysteresis function $f_{an}$ and the intersecting function $\Lambda$. The definitions of the functions $\omega_{\Phi}$ and $f_{an}$ are the same as those given in Section \ref{DuhemCCW}. However the CW intersecting function $\Lambda$ has a different definition than that of the function $\Omega$.

The function $\Lambda:\rline^2\to \rline$ is a {\it CW intersecting function} if $\omega_\Phi(\Lambda(\gamma,v),\gamma,v)=f_{an}(\Lambda(\gamma,v))$ for all $(\gamma,v)\in\rline^2$ and $\Lambda(\gamma,v)\leq v$ whenever $\gamma \geq f_{an}(v)$ and $\Lambda(\gamma,v)>v$ otherwise. Here we assume $\Lambda$ is differentiable. In \cite[Lemma 1]{RUIYUESCL2012} sufficient conditions on $f_1$ and $f_2$ which ensure that such $\Lambda$ exists are $f_{an}$ be monotone increasing and
\begin{eqnarray}
\label{eq_ass_Omega_21}
f_1(\gamma,v) >  \frac{\dd f_{an}(v)}{\dd v}+\epsilon & \textrm{ whenever }& \gamma > f_{an}(v)\ \\
\label{eq_ass_Omega_22}
f_2(\gamma,v) >  \frac{\dd f_{an}(v)}{\dd v}+\epsilon & \textrm{ whenever }& \gamma < f_{an}(v)\
\end{eqnarray}
hold with $\epsilon>0$.

We recall our main results in \cite{RUIYUESCL2012} in the following theorem, which gives the sufficient conditions on $\Phi$ such that it is CW.

\begin{theorem}\cite[Theorem $1$]{RUIYUESCL2012}\label{thm2}
Consider the Duhem hysteresis operator $\Phi$ defined in \rfb{babuskamodel} with $C^1$ functions $f_1, f_2:\rline^2 \to \rline_+$. Let $f_{an}$ be the corresponding anhysteresis function which satisfies \rfb{eq_ass_Omega_21} and \rfb{eq_ass_Omega_22}. Denote by $\Lambda$ the corresponding CW intersecting function. Suppose that for all $(\gamma,v)$ in $\rline^2$, $f_1(\gamma,v)\geq f_2(\gamma,v)$ whenever $\gamma \leq f_{an}(v)$ and $f_1(\gamma,v)< f_2(\gamma,v)$ otherwise. Let the anhysteresis function $f_{an}$ satisfies $f_{an}(0)= 0$. Then $\Phi$ is CW with the storage function $H_{\circlearrowright}: \rline^2 \rightarrow \rline_+$ be given by
\begin{equation}\label{storage_cw}
H_{\circlearrowright}(\gamma,v)=\int_{0}^{\Lambda(\gamma, v)}{f_{an}(\sigma) \dd \sigma}-\int_{v}^{\Lambda(\gamma,v)}{\omega_\Phi(\sigma,\gamma,v) \dd \sigma},
\end{equation}
\end{theorem}
\vspace{0.2cm}
The proof is similar to that of Theorem \ref{CCW_hys} where we use also the result in \cite[Theorem $1$]{RUIYUESCL2012} which shows that $H_{\circlearrowright}$ satisfies \rfb{passivehysteresis_cw}.

An example of the CW hysteresis phenomenon is the friction-induced hysteresis in mechanical system, which has CW behavior from the input (i.e., the relative displacement) to the output (i.e., the friction force). One of the standard model to describe friction-induced hysteresis is the Dahl model \cite{DAHL1976,PADTHE2008}, which is given by
\begin{equation}\label{dahl_duhem}
\dot{y}_{\Phi}(t)=\rho\left |1-\frac{y_{\Phi}(t)}{F_c}\textrm{sgn} (\dot{u}_{\Phi}(t))\right |^{r}\\\textrm{sgn} \left (1-\frac{y_{\Phi}(t)}{F_c}\textrm{sgn}(\dot{u}_{\Phi}(t))\right)\dot{u}_{\Phi}(t),
\end{equation}
where $y_{\Phi}$ denotes the friction force, $u_{\Phi}$ denotes the relative displacement, $F_c>0$ denotes the Coulomb friction force, $\rho >0$ denotes the rest stiffness and $r\geq 0$ is a parameter that determines the shape of the hysteresis loops.

The Dahl model can be described by the Duhem hysteresis operator \rfb{babuskamodel} with
\begin{equation}\label{dahlf1}
f_1(y_{\Phi},u_{\Phi}) = \rho \left | 1-\frac{y_{\Phi}}{F_c}\right |^{r} \textrm{sgn}\left (1-\frac{y_{\Phi}}{F_c}\right ),\ f_2(y_{\Phi},u_{\Phi}) = \rho \left | 1+\frac{y_{\Phi}}{F_c}\right |^{r} \textrm{sgn}\left (1+\frac{y_{\Phi}}{F_c}\right ).
\end{equation}
In Figure \ref{dahl}, we illustrate the behavior of the Dahl model where $F_c=0.75$, $\rho=1.5$ and $r=1$.
\begin{figure}[h]
\centering \psfrag{D}{$F(t)$}\psfrag{u}{$u(t)$}\psfrag{t}{$t$}
 \subfigure[] {{\includegraphics[width=2.4in]{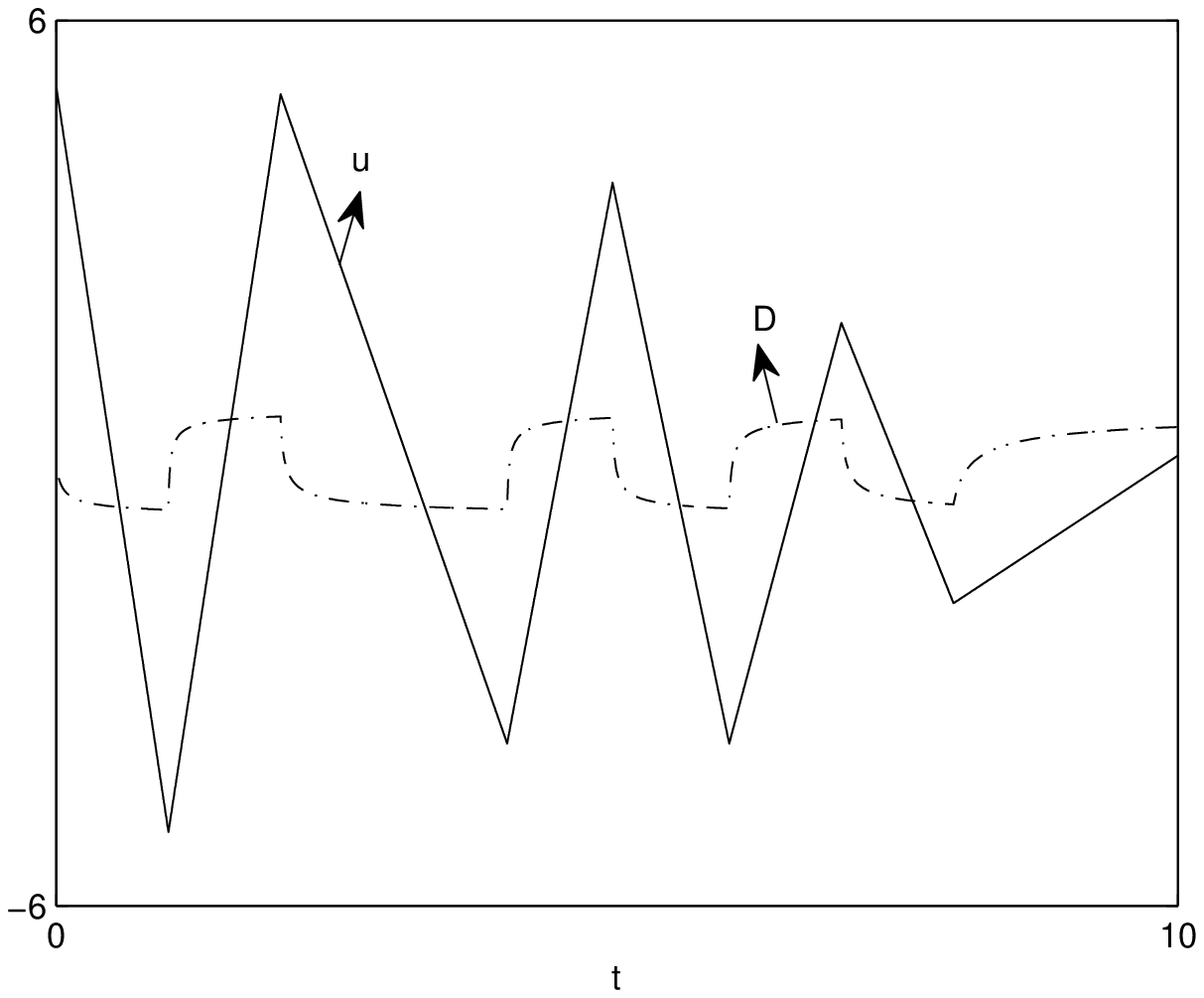}}}
 \subfigure[] {{\includegraphics[width=2.4in]{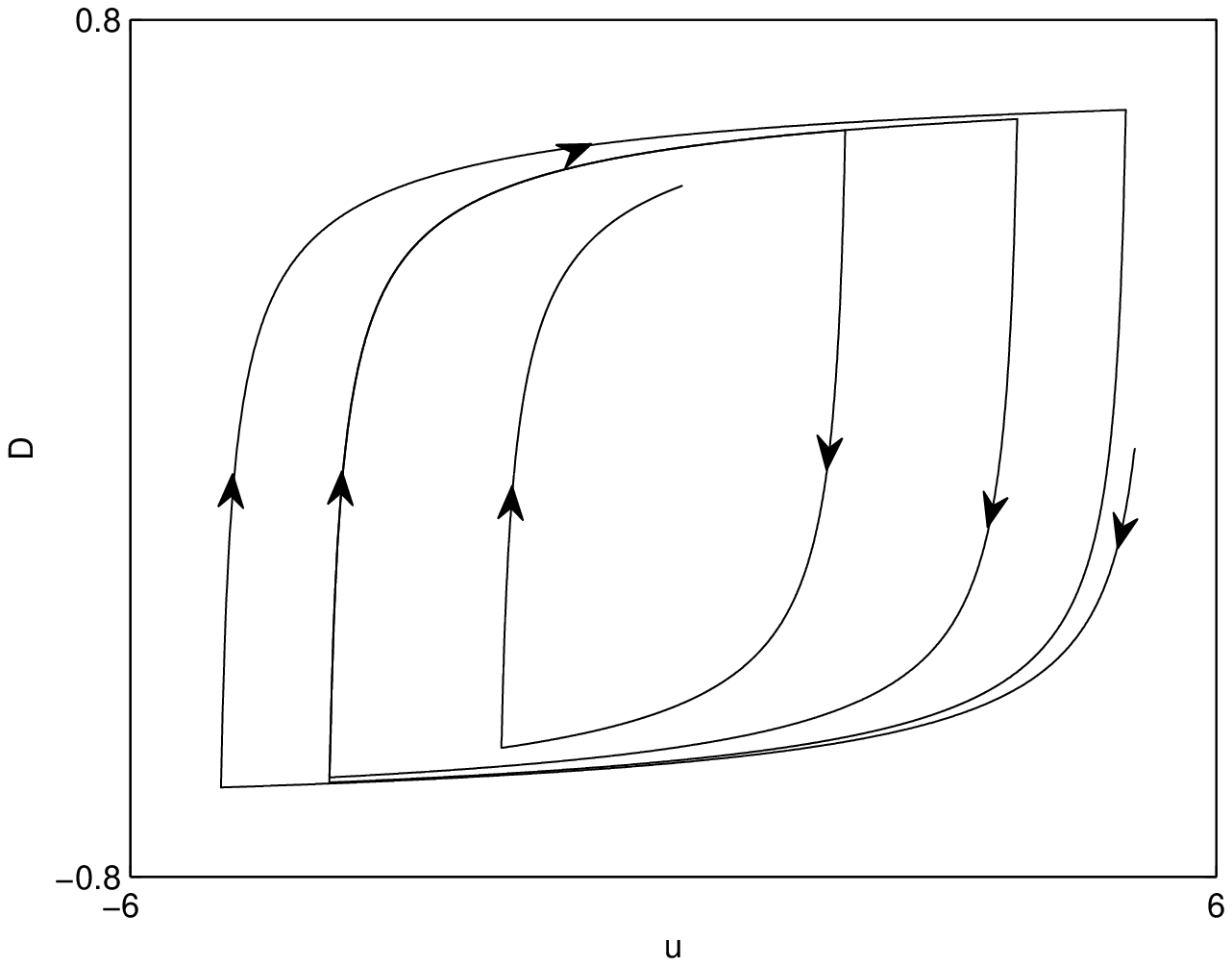}}}
\caption{The input-output dynamics of the Dahl model with $F_c=0.75$, $\rho=1.5$ and $r=1$.}
\label{dahl} \vspace{0.2cm}
\end{figure}

It is immediate to check that $f_1$ and $f_2$ satisfy the hypotheses in \rfb{exis_cond}, which means that for all $u_{\Phi} \in AC(\rline_+)$ and $y_{\Phi}(0)\in \rline$ the solution of \rfb{dahl_duhem} exists for all $t\in \rline_+$. The anhysteresis function of the Dahl model is $f_{an}(u_{\Phi}(t)) = 0$.

Calculating the curve $\omega_{\Phi}$, we have
\begin{equation}
\omega_{\Phi}(\sigma,y_{\Phi}(t),u_{\Phi}(t))=\left \{ \begin{array}{ll} F_c+(y_{\Phi}(t)-F_c)\mathop{e}^{\frac{\rho}{F_c}(u_{\Phi}(t)-\sigma)} \quad \sigma \in[u_{\Phi}(t),\ \infty),\\
                              -F_c+(y_{\Phi}(t)+F_c)\mathop{e}^{\frac{\rho}{F_c}(\sigma -u_{\Phi}(t))}\quad \sigma \in(-\infty,\ u_{\Phi}(t)]. \end{array} \right.
\end{equation}

The CW intersecting function $\Lambda(y_{\Phi}(t),u_{\Phi}(t))$ is given by
\begin{equation}\label{intersectionfun}
\Lambda(y_{\Phi}(t),u_{\Phi}(t))= \left \{ \begin{array}{ll} u_{\Phi}(t)+\frac{F_c}{\rho}{\rm ln}{\frac{F_c}{y_{\Phi}(t)+F_c}}\quad y_{\Phi}(t) \geq 0,\\
                              u_{\Phi}(t)-\frac{F_c}{\rho}{\rm ln}{\frac{-F_c}{y_{\Phi}(t)-F_c}}\quad y_{\Phi}(t) < 0, \end{array} \right.
\end{equation}

\begin{figure}[h]
\centering \psfrag{H}{$H_{\circlearrowright}(y_{\Phi}(t),u_{\Phi}(t))$}\psfrag{u}{$u(t)$}\psfrag{t}{$t$} \psfrag{yp}{$y_{\Phi}(t)$}\psfrag{up}{$u_{\Phi}(t)$}\psfrag{lamb}{$\Lambda(y_{\Phi}(t),u_{\Phi}(t))$} \psfrag{fan}{$f_{an}$}
\includegraphics[width=4in]{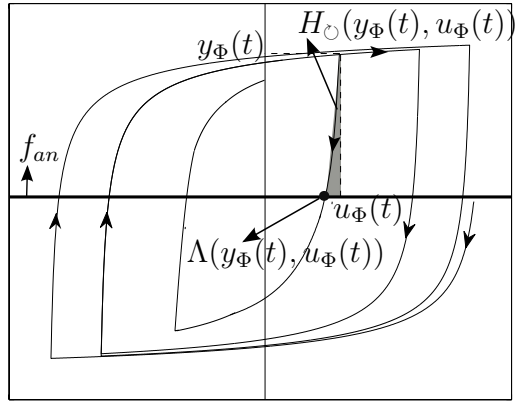}
\caption{Graphical interpretation of the function $H_{\circlearrowright}(y_{\Phi}(t),u_{\Phi}(t))$ of the Dahl model using $f_1$ and $f_2$ as in \rfb{dahlf1} with $\sigma=1$, $F_{c}= 0.5$ and $y_{\Phi_0}=0$.}
\label{Dahl_storage}
\end{figure}
Denoting $u_{\Phi}^*(t)=\Lambda(y_{\Phi}(t),u_{\Phi}(t))$, we can compute explicitly the function $H_{\circlearrowright}$ in \rfb{storage_cw} as follows
%
\begin{equation*}
H_{\circlearrowright}(y_{\Phi}(t),u_{\Phi}(t))= \left \{ \begin{array}{ll} -F_c(u_{\Phi}(t)-u_{\Phi}^*(t))+\frac{F_c}{\rho}(y_{\Phi}(t)+F_c)(1-\mathop{e}^{\frac{\rho}{F_c}(u_{\Phi}^*(t)-u_{\Phi}(t))}) \quad y_{\Phi}(t) \geq 0,\\
                               F_c(u_{\Phi}(t)-u_{\Phi}^*(t))+\frac{F_c}{\rho}(y_{\Phi}(t)-F_c)(\mathop{e}^{\frac{\rho}{F_c}(u_{\Phi}(t)-u_{\Phi}^*(t))}-1)\quad y_{\Phi}(t) < 0. \end{array} \right.
\end{equation*}

The graphical interpretation of $H_{\circlearrowright}$ is shown in Figure \ref{Dahl_storage}, where the value of $H_{\circlearrowright}$ at a given time $t$ is given by the area in grey.

\begin{proposition}\label{proposHcw}
Consider a Duhem operator $\Phi$ satisfying the hypotheses in Theorem \ref{thm2}. Suppose that $f_{an}$ is monotone increasing and $f_{an}(0)=0$. Then the function $H_{\circlearrowright}(\cdot,v)$ (where $H_{\circlearrowright}$ is as in \rfb{storage_cw}) is radially unbounded for every $v$.
\end{proposition}
\begin{proof}
To show the properness of $H_{\circlearrowright}(\cdot,v)$ for any given $v$, we first consider the case $\gamma \geq f_{an}(v)$. Since the Duhem operator $\Phi$ satisfies the hypotheses in Theorem \ref{thm2}, the function $H_{\circlearrowright}$ is nonnegative. Thus, using \rfb{storage_cw} and since $\Lambda(f_{an}(v),v)=v$ we have
\begin{align*}
H_{\circlearrowright}(\gamma,v) &\geq H_{\circlearrowright}(\gamma,v)-H_{\circlearrowright}(f_{an}(v),v)\\&= \int_{\Lambda(f_{an}(v),v)}^{\Lambda(\gamma,v)}{f_{an}(\sigma) \dd \sigma}-\int_{v}^{\Lambda(\gamma,v)}{\omega_{\Phi}(\sigma,\gamma,v)\dd \sigma} + \int_{v}^{\Lambda(f_{an}(v),v)}{\omega_{\Phi}(\sigma,\gamma,v)\dd \sigma}\\
&= \int_{v}^{\Lambda(\gamma,v)}{f_{an}(\sigma)-\omega_{\Phi}(\sigma,\gamma,v)\dd \sigma}.
\end{align*}
By the definition of the CW intersecting function $\Lambda$, $\gamma \geq f_{an}(v)$ implies that $\Lambda(\gamma,v)<v$. Using the monotonicity of $\omega_{\Phi}$, $\omega_{\Phi}(\sigma,\gamma,v)\leq \gamma$ for all $\sigma<v$, and thus, it follows from the above inequality that
\begin{align*}
H_{\circlearrowright}(\gamma,v)&\geq \int_{v}^{\Lambda(\gamma,v)}{f_{an}(\sigma)-\omega_{\Phi}(\sigma,\gamma,v)\dd \sigma}\\
&\geq \int_{v}^{\Lambda(\gamma,v)}{f_{an}(\sigma)-\gamma \dd \sigma},
\end{align*}
Now let us fix $\bar{\gamma}$ s.t. $f_{an}(v)<\bar{\gamma} < \gamma$. Since $\omega_{\Phi}(\sigma,\bar{\gamma},v)< \omega_{\Phi}(\sigma,\gamma,v)$ for all $\sigma < v$ and using the monotonicity of $f_{an}$, we have that $\Lambda(\gamma,v) < \bar{\Lambda}$, where $\bar{\Lambda}=\Lambda(\bar{\gamma},v)$. Therefore,
\begin{align*}
H_{\circlearrowright}(\gamma,v) &\geq \int_{v}^{\Lambda(\gamma,v)}{ f_{an}(\sigma)-\gamma  \dd \sigma} \geq \int_{v}^{\bar{\Lambda}}{f_{an}(\sigma)-\gamma  \dd \sigma}\\
&\geq  \int_{v}^{\bar{\Lambda}}{f_{an}(v)-\gamma  \dd \sigma}=(\gamma-c)(v-\bar{\Lambda})>0
\end{align*}
where $c:=f_{an}(v)$ and $\bar{\Lambda}=\Lambda(\bar{\gamma},v)$ for any given $v $. Hence, it implies that for every $v$, $H_{\circlearrowright}(\gamma, v)\rightarrow \infty$ as $\gamma \rightarrow \infty$.

We can apply similar arguments to show that for every $v$, $H_{\circlearrowright}(\gamma, v)\rightarrow \infty$ as $\gamma \rightarrow -\infty$ by evaluating the case when $\gamma < f_{an}(v)$.
\end{proof}
\section{Linear system with CCW Duhem hysteresis}\label{feedback_CCW}
In this section we analyze the stability of a feedback interconnection of a linear system and a CCW Duhem hysteresis operator. The stability of the closed-loop system is analyzed by exploiting the CCW or CW properties of each subsystem.
\begin{theorem}\label{SOCCW-CCW}
Consider a positive feedback interconnection of a minimal single-input single-output linear system and a Duhem operator $\Phi$ as shown in Figure \ref{blockscheme} satisfying the hypotheses in Theorem \ref{CCW_hys} as follows
\begin{equation} \label{pffeedback_CCW}
\left.\begin{array}{rl}
  \calP :& \begin{array}{rl}     \dot{x}&=Ax+Bu,\\
          y&=Cx,\end{array}\\
   {\bf \Phi}:&   \dot{y}_{\Phi}=f_1(y_{\Phi}(t),u_{\Phi}(t))\dot u_{\Phi+}(t) + f_2(y_{\Phi}(t),u_{\Phi}(t))\dot u_{\Phi-}(t),\\
          &u= y_{\Phi},  \ u_{\Phi}=y,
        \end{array}\right\}
\end{equation}
where $A\in \rline^{n\times n}$, $B\in\rline^{n\times 1}$ and $C\in\rline^{1\times n}$. Let $\varepsilon:= (CB)^{-1}$ where we assume that $CB>0$. Suppose that there exist $\xi>0$ and $Q=Q^T >0$ such that
\begin{align}
\frac{1}{2}(A^TQ+QA)+\varepsilon A^TC^TCA &\leq 0, \label{LMI11} \\
QB + A^TC^T &= 0, \label{LMI21}\\
Q-\xi C^TC&>0, \label{LMI3}
\end{align}
hold  and the anhysteresis function $f_{an}$ satisfies $(f_{an}(v)-\xi v)v \leq 0$ for all $v \in \rline$ (i.e. $f_{an}$ belongs to the sector $[0, \xi]$). Then for every initial condition $(x(0),y_{\Phi}(0))$, the state trajectory of the closed-loop system \rfb{pffeedback_CCW} is bounded and converges to the largest invariant set in $\{(x,y_{\Phi})|CAx+CBy_{\Phi}=0 \}$.
\end{theorem}

\begin{proof}
Using $V(x)=\frac{1}{2}x^TQx$ and \rfb{LMI11} and \rfb{LMI21}, it can be checked that
\begin{align*}
\dot{V} &= \frac{1}{2}x^T(A^TQ+QA)x+x^TQBu\\
&\leq -\varepsilon x^TA^TC^TCAx -x^TA^TC^Tu\\
&=(u^TB^TC^T+x^TA^TC^T)u-\varepsilon (CAx+CBu)^T(CAx+CBu)\\
&= \langle\dot{y}, u \rangle - \varepsilon \dot{y}^2.
\end{align*}
It follows from Lemma \ref{lemma1} that the linear system is S-CCW.

By the assumptions of the theorem, the Duhem operator $\Phi$ is also CCW with the storage function $H_{\circlearrowleft} : \rline^2 \rightarrow \rline_+$ as given in \rfb{storage}.

Now let $H_{cl}(x,y_{\Phi})=V(x) + H_{\circlearrowleft}(y_{\Phi},Cx)-Cxy_{\Phi}$ be the Lyapunov function of the interconnected system \rfb{pffeedback_CCW}. We show first that $H_{cl}$  is lower bounded. Substituting the representation of $V$ and $H_{\circlearrowleft}$, we have
\begin{align}\nonumber
H_{cl}&=\frac{1}{2}x^TQx+zCx-\int_0^{Cx}{\omega_\Phi(\sigma,y_{\Phi},Cx) \ \dd\sigma} + \int_0^{\Omega(y_{\Phi},Cx)}{\omega_\Phi(\sigma,y_{\Phi},Cx)\dd\sigma}\\\nonumber
      &- \int_0^{\Omega(y_{\Phi},Cx)}{f_{an}(\sigma)\dd\sigma}-Cxy_{\Phi}\\\nonumber
      &=\frac{1}{2}x^TQx-\int_0^{Cx}{f_{an}(\sigma)\dd\sigma}+\int_{Cx}^{\Omega(y_{\Phi},Cx)}{\omega_\Phi(\sigma,y_{\Phi},Cx)-f_{an}(\sigma)\dd\sigma}\\\nonumber
      & \geq \frac{1}{2}x^TQx-\int_0^{Cx}{(f_{an}(\sigma)-\xi \sigma)\dd\sigma}-\int_0^{Cx}{\xi \sigma \dd \sigma}+\int_{Cx}^{\Omega(y_{\Phi},Cx)}{\omega_\Phi(\sigma,y_{\Phi},Cx)-f_{an}(\sigma)\dd\sigma}\\\label{ccw_hcl_rad}
      &\geq \frac{1}{2}x^T(Q-\xi C^TC)x+\int_{Cx}^{\Omega(y_{\Phi},Cx)}{\omega_\Phi(\sigma,y_{\Phi},Cx)-f_{an}(\sigma)\dd\sigma}.
      \end{align}
%
where the last inequality is due to the sector condition on $f_{an}$. In the following, we will prove that the last term on the RHS of \rfb{ccw_hcl_rad} is lower bounded. Notice that since $f_1 \geq0$, $f_2\geq 0$, \rfb{eq_ass_Omega_11} and \rfb{eq_ass_Omega_12} imply that $\frac{\dd f_{an}(v)}{\dd v}>\epsilon$ for some $\epsilon>0$. Hence $f_{an}$ is strictly increasing and invertible.

Consider the case when $y_{\Phi} \geq f_{an}(Cx)$ which implies also that $\Omega(y_{\Phi},Cx)\geq f_{an}^{-1}(y_{\Phi})\geq Cx$ by the definition of $\Omega$. Using the monotonicity of $\omega_{\Phi}$ we have
\begin{equation*}
\int_{Cx}^{\Omega(y_{\Phi},Cx)}{\omega_\Phi(\sigma,y_{\Phi},Cx)-f_{an}(\sigma)\dd\sigma} \geq \int_{Cx}^{\Omega(y_{\Phi},Cx)}{y_{\Phi} -f_{an}(\sigma)\dd \sigma}\geq \int_{Cx}^{f_{an}^{-1}(y_{\Phi})}{y_{\Phi} -f_{an}(\sigma)\dd \sigma}.
\end{equation*}
Define $V(y_{\Phi},Cx):=\int_{Cx}^{f_{an}^{-1}(y_{\Phi})}{y_{\Phi} -f_{an}(\sigma)\dd \sigma}$ and let $c:=\frac{y_{\Phi}-f_{an}(Cx)}{2}+f_{an}(Cx)$. It follows that $f_{an}^{-1}(y_{\Phi}) \geq f_{an}^{-1}(c)$ and $f_{an}(\sigma) \leq c$ for all $\sigma \in [Cx, f_{an}^{-1}(c)]$. Therefore
\begin{align*}
V(y_{\Phi},Cx)&\geq \int_{Cx}^{f_{an}^{-1}(c)}{y_{\Phi} -f_{an}(\sigma)\dd \sigma}\geq \int_{Cx}^{f_{an}^{-1}(c)}{y_{\Phi} -c\ \dd \sigma}\\
 &= \frac{1}{2}(y_{\Phi}-f_{an}(Cx))(f_{an}^{-1}(c)-Cx)\geq 0.
\end{align*}
Thus, $\int_{Cx}^{\Omega(y_{\Phi},Cx)}{y_{\Phi} -f_{an}(\sigma)\dd \sigma}$ is lower bounded by $V(y_{\Phi},Cx)$ which is positive definite (it is equal to zero only if $y_{\Phi}=f_{an}(Cx)$) and $V(y_{\Phi},Cx)\rightarrow \infty$ as $y_{\Phi}\rightarrow \infty$.

When $y_{\Phi}<f_{an}(Cx)$, we can obtain the same result where $\int_{Cx}^{\Omega(y_{\Phi},Cx)}{y_{\Phi} -f_{an}(\sigma)\dd \sigma}$ is lower bounded by $V(y_{\Phi},Cx)$ which is positive definite and $V(y_{\Phi},Cx)\rightarrow \infty$ as $y_{\Phi}\rightarrow -\infty$.

Therefore, using \rfb{ccw_hcl_rad}, we have
\begin{equation*}
 H_{cl}\geq \frac{1}{2}x^T(Q-\xi C^TC)x+V(y_{\Phi},Cx),
\end{equation*}
which is radially unbounded.

Now computing the time derivative of $H_{cl}$, we obtain
\begin{equation*}
\dot{H}_{cl} = \dot{V} + \dot{H}_{\circlearrowleft}-C\dot{x}y_{\Phi}-Cx\dot{y}_{\Phi} \leq -\varepsilon \dot{y}^2.
\end{equation*}
This inequality together with the radially unboundedness of $H_{cl}$ imply that the trajectory $(x,y_{\Phi})$ is bounded. Using the Lasalle's invariance principle, we conclude that the trajectory $(x,y_{\Phi})$ of \rfb{pffeedback_CCW} converges to the largest invariant set contained in $M:=\{(x,y_{\Phi})\in \rline^n \times \rline|CAx+CBy_{\Phi}=0\}$.



\end{proof}

We illustrate Theorem \ref{SOCCW-CCW} in the following simple example.

\begin{example}\label{ex2}
Consider
\begin{equation*}
\begin{array}{rl}
\calP :& \dot{x}  = -x+u,\ y  = x,\\
{\bf \Phi}:& \dot{y}_{\Phi}=f_1(y_{\Phi}(t),u_{\Phi}(t))\dot u_{\Phi+}(t) + f_2(y_{\Phi}(t),u_{\Phi}(t))\dot u_{\Phi-}(t),\\
&u= y_{\Phi},  \ u_{\Phi}=y,
\end{array}
\end{equation*}
where $x(t)\in\R$ and the functions $f_1$, $f_2$ satisfy the hypotheses in Theorem \ref{CCW_hys}. Using $Q=1$, it can be checked that \rfb{LMI11} $-$ \rfb{LMI3} hold. Using $H_{cl}$ as in the proof of Theorem \ref{SOCCW-CCW},
let us define $H_{cl}(x,y_{\Phi})=\frac{1}{2}x^2+H_{\circlearrowleft}(y_{\Phi},y)-yy_{\Phi}$ and a routine computation shows that
\begin{align*}
\dot H_{cl} & \leq \dot{y}y_{\Phi} +  \dot{\overbrace{\Phi(y)}}y -\dot{y}y_{\Phi} - y\dot{y}_{\Phi}-\dot{y}^2\\
& = -(-x+y_{\Phi})^2.
\end{align*}
Note that $Q=1$, $C=1$, so that \rfb{LMI41} holds for $\xi < 1$. This means that the result in Theorem \ref{SOCCW-CCW} holds if the anhysteresis function $f_{an}$ satisfies $(f_{an}(v)-\xi v)v \leq 0$, for all $v \in \rline$ and $\xi <1$. In other words, $f_{an}$ should belong to the sector $[0,\xi]$ for the stability of the closed-loop system.
\end{example} $\hfill \triangle$
\vspace{0.2cm}

The result in Theorem \ref{SOCCW-CCW} deals with a positive feedback interconnection of a linear system and a Duhem hysteresis operator. This is motivated by the study of an interconnection between counterclockwise systems as studied in \cite{ANGELI2006} for the general case and in \cite{Petersen2010} for the linear case. In the following result, we consider the other case where a negative feedback is used instead.

\begin{theorem}\label{SOCW-CCW}
Consider a negative feedback interconnection of a minimal single-input single-output linear system and a Duhem operator $\Phi$ as shown in Figure \ref{blockscheme} satisfying the hypotheses in Theorem \ref{CCW_hys} as follows
\begin{equation} \label{negfeedback}
\left.\begin{array}{rl}
 \calP :&\begin{array}{rl}   \dot{x}&=Ax+Bu,\\
          y&=Cx+Du,\end{array}\\
  {\bf \Phi}:& \dot{y}_{\Phi}=f_1(y_{\Phi}(t),u_{\Phi}(t))\dot u_{\Phi+}(t) + f_2(y_{\Phi}(t),u_{\Phi}(t))\dot u_{\Phi-}(t),\\
          &u= - y_{\Phi},  \ u_{\Phi}=y,
        \end{array}\right\}
\end{equation}
where $A\in \rline^{n\times n}$, $B\in\rline^{n\times 1}$, $C\in\rline^{1\times n}$ and $D\in\rline$. Assume that there exist $P=P^T>0$, $L$ and $\delta >0$ such that the following linear matrix inequalities (LMI)
\begin{equation}\label{LMI1}
P\sbm{
    1 \\
    0^{n \times 1}}
 = \sbm{
     D \\
     C^T},
 \\
\end{equation}
\begin{equation}\label{LMI2}
  \frac{1}{2} \left (P\sbm{
     0 & 0^{n\times n} \\
     B & A } + \sbm{
                                 0 & B^T \\
                                 0^{n\times n} & A^T
                               }P \right )
                                +\delta  L^TL \leq 0,
\end{equation}
hold. Then for every initial condition $(x(0),y_{\Phi}(0))$, the state trajectory of the closed-loop system \rfb{negfeedback} is bounded and converges to the largest invariant set in $\{ (x,y_{\Phi})|L \sbm{-y_{\Phi} \\x}=0 \}$.
\end{theorem}

\begin{proof}
By the assumptions of the theorem, the Duhem operator $\Phi$ is CCW with the function $H_{\circlearrowleft}:\rline^2 \rightarrow \rline_+$ as given in \rfb{storage}.

Define the extended state space of the linear system in \rfb{negfeedback} by
\begin{equation}\label{exstatespace2}
 \left.\begin{array}{rl}
  \calP_{ext}: & \begin{array}{rl}   \dot{w} &=q,\\
         \dot{x}&=Ax+Bw, \\
          y&=Cx+Dw,\end{array}
        \end{array}\right\}
\end{equation}
where $w=u$.

Using $V = \frac{1}{2}[w \ x^T]^TP \left [ \begin{array}{c}
                     w \\
                     x
                   \end{array}\right ]$, a routine computation shows that
\begin{equation*}
\dot{V}    =\frac{1}{2} [\begin{array}{cc}
            w & x^T
          \end{array}]\left ( \left [\begin{array}{cc}
                        0 & B^T \\
                        0^{n \times n} & A^T
                      \end{array} \right ] P \right . + P \left . \left [\begin{array}{cc}
                        0 & 0^{n \times n}\\
                        B & A
                      \end{array} \right ]\right ) \left [\begin{array}{c}
                                                    w \\
                                                    x
                                                  \end{array}\right ]
                                                   + [\begin{array}{cc}
                                                              w & x^T
                                                                 \end{array}] P\left [\begin{array}{c}
                                                                                              1 \\
                                                                                               0^{n \times 1}
                                                                                            \end{array} \right ]q.
\end{equation*}
Using \rfb{LMI1} and \rfb{LMI2},
\begin{equation}\label{soccw_neg}
\dot{V} \leq \langle y, q\rangle - \delta \left \|L \left [\begin{array}{c}
                                                       -y_{\Phi} \\
                                                       x
                                                     \end{array}\right ]\right \|^2.
\end{equation}
This inequality \rfb{soccw_neg} with $q=\dot{u}$ (by the relation in \rfb{exstatespace2}) implies that the linear system defined in \rfb{negfeedback} is CW.

Now take $H_{cl}(x,y_{\Phi})=H_{\circlearrowleft}(y_{\Phi},Cx-Dy_{\Phi}) + V(x,y_{\Phi})$ as the Lyapunov function of the interconnected system \rfb{negfeedback}, where $H_{cl}$ is radially unbounded by the non-negativity of $H_{\circlearrowleft}$ and the properness of $V$. It is straightforward to see that
\begin{align}\nonumber
\dot{H}_{cl} &= \dot{H}_{\circlearrowleft} + \dot{V}, \\\nonumber
           &\leq \langle y, \dot{u}\rangle +\langle \dot{y}_{\Phi}, u_{\Phi} \rangle - \delta \left \|L \left [\begin{array}{c}
                                                       -y_{\Phi} \\
                                                       x
                                                     \end{array}\right ]\right \|^2,\\\label{theorem_SOCCW}
           &=- \delta \left \|L \left [\begin{array}{c}
                                                       -y_{\Phi} \\
                                                       x
                                                     \end{array}\right ]\right \|^2,
\end{align}
where the last equation is due to the interconnection conditions $u = - y_{\Phi}$ and $y = u_{\Phi}$.
It follows from \rfb{theorem_SOCCW} and from the radial unboundedness (or properness) of $H_{cl}$, the signals $x$ and $y_{\Phi}$ are bounded.

Based on the Lasalle's invariance principle \cite{LOGEMANN2004}, the semiflow $(x,y_{\Phi})$ of \rfb{negfeedback} converges to the largest invariant set contained in $M:= \{(x,y_{\Phi})\in \rline^n \times \rline|L \sbm{-y_{\Phi} \\x}=0\}$.
\end{proof}

To illustrate Theorem \ref{SOCW-CCW}, let us consider the following simple example.
\begin{example} \label{ex1}
Let
\begin{equation*}
\begin{array}{rl}
\calP: & \begin{array}{rl}\dot{x} & = -3x+u,\\
y & = -2x+u,\end{array}\\
{\bf \Phi}:& \dot{y}_{\Phi}=f_1(y_{\Phi}(t),u_{\Phi}(t))\dot u_{\Phi+}(t) + f_2(y_{\Phi}(t),u_{\Phi}(t))\dot u_{\Phi-}(t),\\
 &u= -y_{\Phi},  \ u_{\Phi}=y,
\end{array}
\end{equation*}
where $x(t)\in\R$ and the functions $f_1$, $f_2$ satisfy the hypotheses in Theorem \ref{CCW_hys}. By using $P=\sbm{1 & -2 \\-2 & 6}$, it can be checked that \rfb{LMI1} $-$ \rfb{LMI2} hold.
Following the same construction as in the proof of Theorem \ref{SOCW-CCW}, we define $H_{cl}(x,y_{\Phi})=\frac{1}{2}x^TPx+H_{\circlearrowleft}(-y_{\Phi},-2x+y_{\Phi})$ and a routine computation shows that
\begin{align*}
\dot H_{cl} & \leq -2(-3x+y_{\Phi})^2 + y\dot{y}_{\Phi} - \dot{\overbrace{\Phi(y)}}y \\
& = -2(-3x+y_{\Phi})^2.
\end{align*}
Thus, we can conclude that $(x,y_{\Phi})$ converges to the invariant set where $x=\frac{1}{3}y_{\Phi}$.
\end{example} $\hfill \triangle$

\section{Linear system with CW Duhem hysteresis}
Dual to the result that we present in the previous section, the feedback interconnection of a linear system and a CW Duhem hysteresis is considered in this section.

\begin{theorem}\label{SOCCW-CW}
Consider a negative feedback interconnection of a minimal single-input single-output linear system and a Duhem operator $\Phi$ as shown in Figure \ref{blockscheme} satisfying the hypotheses in Theorem \ref{thm2} as follows
\begin{equation} \label{pffeedback}
\left.\begin{array}{rl}
  \calP:& \begin{array}{rl} \dot{x}&=Ax+Bu,\\
          y&=Cx,\end{array}\\
  {\bf \Phi}: &       \dot{y}_{\Phi}=f_1(y_{\Phi}(t),u_{\Phi}(t))\dot u_{\Phi+}(t) + f_2(y_{\Phi}(t),u_{\Phi}(t))\dot u_{\Phi-}(t),\\
          &u= -y_{\Phi},  \ u_{\Phi}=y,
        \end{array}\right\}
\end{equation}
where $A\in \rline^{n\times n}$, $B\in\rline^{n\times 1}$ and $C\in\rline^{1\times n}$. Let $\varepsilon:= (CB)^{-1}$ where we assume $CB>0$ and assume that there exist $Q=Q^T >0$ such that
\begin{align}
\frac{1}{2}(A^TQ+QA)+\varepsilon A^TC^TCA &\leq 0, \label{LMI31} \\
QB + A^TC^T &= 0, \label{LMI32}
\end{align}
hold. Then for every initial condition $(x(0),y_{\Phi}(0))$, the state trajectory of the closed-loop system \rfb{pffeedback} is bounded and converges to the largest invariant set in $\{(x,y_{\Phi})|CAx-CBy_{\Phi}=0 \}$.
\end{theorem}
\vspace{0.2cm}
\begin{proof}
Let $V(x)=\frac{1}{2}x^TQx$, and using \rfb{LMI31}$-$\rfb{LMI32}, it can be checked that
\begin{align*}
\dot{V} &= \frac{1}{2}x^T(A^TQ+QA)x+x^TQBu\\
&\leq \langle\dot{y}, u \rangle - \varepsilon \dot{y}^2.
\end{align*}
It follows from Lemma \ref{lemma1} that the linear system is S-CCW.

By the assumptions of the theorem, the Duhem operator $\Phi$ is CW with the function $H_{\circlearrowright} : \rline^2 \rightarrow \rline_+$ as given in \rfb{storage_cw}.

Now let $H_{cl}(x,y_{\Phi})=V(x) +H_{\circlearrowright}(y_{\Phi},Cx)$ as the Lyapunov function of the system \rfb{pffeedback}. According to Proposition \ref{proposHcw}, $H_{\circlearrowright}(y_{\Phi}, Cx)$ is radially unbounded for every $x$, which implies that $H_{cl}(x,y_{\Phi})$ is radially unbounded.

Computing the time derivative of $H_{cl}$, we obtain
\begin{equation*}
\dot{H}_{cl} = \dot{V} + \dot{H}_{\circlearrowleft} \leq -\varepsilon \dot{y}^2.
\end{equation*}
This inequality together with the radially unboundedness of $H_{cl}$ imply that the trajectory $(x,y_{\Phi})$ is bounded. Using the Lasalle's invariance principle, we conclude that the trajectory $(x,y_{\Phi})$ of \rfb{pffeedback} converges to the largest invariant set contained in $M:=\{(x,y_{\Phi})\in \rline^n \times \rline|CAx-CBy_{\Phi}=0\}$.

\end{proof}
To illustrate Theorem \ref{SOCCW-CW} we could use the same linear system as given in the Example \ref{ex2}.
\begin{example}
\begin{equation*}
 \begin{array}{rl}
\calP:& \dot{x}  = -x+u,\ y  = x,\\
{\bf \Phi}: & \dot{y}_{\Phi}=f_1(y_{\Phi}(t),u_{\Phi}(t))\dot u_{\Phi+}(t) + f_2(y_{\Phi}(t),u_{\Phi}(t))\dot u_{\Phi-}(t),\\
&u= -y_{\Phi},  \ u_{\Phi}=y,\end{array}
\end{equation*}
where $x(t)\in\R$ and the functions $f_1$, $f_2$ satisfy the hypotheses in Theorem \ref{thm2}. Using $Q=1$, it can be checked that \rfb{LMI31} and \rfb{LMI32} hold. Define $H_{cl}(x,y_{\Phi})=\frac{1}{2}x^2+H_{\circlearrowright}(y_{\Phi},y)$, a routine computation shows that
\begin{align*}
\dot H_{cl} & \leq \dot{y}y_{\Phi} -  \dot{y}y_{\Phi} -\dot{y}^2\\
& = -(-x+y_{\Phi})^2,
\end{align*}
which implies that $(x,y_{\Phi})$ converges to the invariant set where $x=y_{\Phi}$.
\end{example} $\hfill \triangle$
\vspace{0.2cm}

\begin{theorem}\label{CW-CW}
Consider a positive feedback interconnection of a minimal single-input single-output linear system and a Duhem operator $\Phi$ as shown in Figure \ref{blockscheme} satisfying the hypotheses in Theorem \ref{thm2} as follows
\begin{equation} \label{pffeedbackCW}
\left. \begin{array}{rl}
\calP: &\begin{array}{rl}  \dot{x}&=Ax+Bu,\\
          y&=Cx+Du,\end{array}\\
 {\bf \Phi}: &  \dot{y}_{\Phi}=f_1(y_{\Phi}(t),u_{\Phi}(t))\dot u_{\Phi+}(t) + f_2(y_{\Phi}(t),u_{\Phi}(t))\dot u_{\Phi-}(t),\\
          &u=  y_{\Phi},  \ u_{\Phi}=y,
        \end{array}\right \}
\end{equation}
where $A\in \rline^{n\times n}$, $B\in\rline^{n\times 1}$, $C\in\rline^{1\times n}$ and $D\in\rline$. Assume that there exist $P$, $L$, $\delta$ and $\eta >0$ such that $P=P^T>\eta \sbm{D^2&DC\\C^TD&C^TC}\geq 0$ and the following linear matrix inequalities (LMI)
\begin{equation}\label{LMI41}
P\sbm{
    1 \\
    0^{n \times 1}}
 = \sbm{
     D \\
     C^T},
 \\
\end{equation}
\begin{equation}\label{LMI42}
  \frac{1}{2} \left (P\sbm{
     0 & 0^{n\times n} \\
     B & A } + \sbm{
                                 0 & B^T \\
                                 0^{n\times n} & A^T
                               }P \right )
                                +\delta L^TL \leq 0,
\end{equation}
hold. Assume further that $f_1(\gamma, v) \leq \frac{\eta}{2}$ and $f_2(\gamma, v)\leq \frac{\eta}{2}$ for all $(\gamma,v)\in \rline^2$. Then for every initial condition $(x(0),y_{\Phi}(0))$, the state trajectory of  the closed-loop system \rfb{pffeedbackCW} is bounded and converges to the largest invariant set in $\{ (x,y_{\Phi})|L \sbm{y_{\Phi} \\x}=0 \}$.
\end{theorem}

\begin{proof}
Define an extended system $\calP_{ext}$ as in \rfb{exstatespace2} and let $V(w,x)= \frac{1}{2}\sbm{w&\ x^T}P\sbm{w\\x}$. Using \rfb{LMI41}, \rfb{LMI42} and \rfb{exstatespace2}, we have
\begin{equation}\label{scw}
\dot{V} \leq \langle y, \dot{u}\rangle - \delta \left \|L \left [\begin{array}{c}
                                                       y_{\Phi} \\
                                                       x
                                                     \end{array}\right ]\right \|^2.
\end{equation}
Equation \rfb{scw} indicates that the linear system is CW.

Next we take $H_{cl}(x,y_{\Phi})=H_{\circlearrowright}(y_{\Phi},y)+V(y_{\Phi},x)-yy_{\Phi}$ as the Lyapunov function of the interconnected system. We will show first that $H_{cl}$ is lower bounded. Using the definition of $V$ as above and $H_{\circlearrowright}$ as in \rfb{storage_cw}, we have
\begin{equation*}
H_{cl}= \frac{1}{2}\bbm{w&\ x^T}P\bbm{w\\x}+\int^{\Lambda(y_{\Phi},u_{\Phi})}_0{f_{an}(\sigma)-\omega_{\Phi}(\sigma,y_{\Phi},u_{\Phi})\dd \sigma}+ \int^{u_{\Phi}}_0{\omega_{\Phi}(\sigma, y_{\Phi}, u_{\Phi}) \dd \sigma} -u_{\Phi}y_{\Phi}.
\end{equation*}
Since $u_{\Phi}=y$ (by the interconnection), $u_{\Phi}^2=\sbm{w&\ x^T}\sbm{D^2&DC\\C^TD&C^TC}\sbm{w\\x}$. By the assumption on $P$, there exists $\eta, \varepsilon > 0$ such that $P-\eta\sbm{D^2&DC\\C^TD&C^TC}>\varepsilon I$. Then
\begin{align}\nonumber
H_{cl}&=  \frac{1}{2}\bbm{w&\ x^T}\left(P-\eta\bbm{D^2&DC\\C^TD&C^TC}\right)\bbm{w\\x}+ \frac{\eta}{2}u_{\Phi}^2+\int^{\Lambda(y_{\Phi},u_{\Phi})}_0{f_{an}(\sigma)-\omega_{\Phi}(\sigma,y_{\Phi},u_{\Phi})\dd \sigma} \\\nonumber
&+\int^{u_{\Phi}}_0{\omega_{\Phi}(\sigma, y_{\Phi}, u_{\Phi}) \dd \sigma} -u_{\Phi}y_{\Phi}\\ \label{lbcw-cw}
& \geq \frac{\varepsilon}{2}\left \|\bbm{w\\x}\right \|^2+\int^{\Lambda(y_{\Phi},u_{\Phi})}_0{f_{an}(\sigma)-\omega_{\Phi}(\sigma,y_{\Phi},u_{\Phi})\dd \sigma}+ \int^{u_{\Phi}}_0{\left(\omega_{\Phi}(\sigma, y_{\Phi}, u_{\Phi})-y_{\Phi}+\frac{\eta}{2}u_{\Phi}\right)\dd \sigma}
\end{align}
It can be checked that the second term of \rfb{lbcw-cw} is nonnegative. Indeed, it follows from the property of the CW intersecting function $\Lambda$ that if $\Lambda(y_{\Phi},u_{\Phi})\geq 0$ we have that $f_{an}(\sigma) \geq \omega_{\Phi}(\sigma,y_{\Phi},u_{\Phi})$ for all $\sigma \in [0, \Lambda(y_{\Phi},u_{\Phi})]$ and if $\Lambda(y_{\Phi},u_{\Phi})< 0$ then $f_{an}(\sigma) \leq \omega_{\Phi}(\sigma,y_{\Phi},u_{\Phi})$ for all $\sigma \in [\Lambda(y_{\Phi},u_{\Phi}), 0]$.

To check whether the last term of \rfb{lbcw-cw} is lower bounded, we use the definition of $\omega_{\Phi}$ given in the Section \ref{DuhemCCW}. Consider the case $u_{\Phi} \geq 0$. Using the definition of $\omega_{\Phi}$ in \rfb{wcurve}, the last term of \rfb{lbcw-cw} can be written by
\begin{align*}
&\int^{u_{\Phi}}_0{\left(\omega_{\Phi}(\sigma, y_{\Phi}, u_{\Phi})-y_{\Phi}+\frac{\eta}{2}u_{\Phi}\right) \dd \sigma}\\
&\ =\int^{u_{\Phi}}_0{\left(y_{\Phi}+\int^{\sigma}_{u_{\Phi}}{f_2(\omega_{\Phi}(s,y_{\Phi},u_{\Phi}),s) \dd s}\right) \dd \sigma}+ \int_{0}^{u_{\Phi}}{\frac{\eta}{2}u_{\Phi}-y_{\Phi}\dd \sigma }\\
& \ =\int^{u_{\Phi}}_0{\int^{u_{\Phi}}_{\sigma}{\frac{\eta}{2}-f_2(\omega_{\Phi}(s,y_{\Phi},u_{\Phi}),s) \dd s} \dd \sigma}+\frac{\eta}{4}u_{\Phi}^2\geq 0,
\end{align*}
where the last inequality is due to fact that $f_2(\gamma, v) \leq \frac{\eta}{2}$ for all $(\gamma,v)\in \rline^2$. In a similar way, we can obtain the non-negativity of $\int^{u_{\Phi}}_{0}{(\omega_{\Phi}(\sigma, y_{\Phi}, u_{\Phi})-y_{\Phi}+\frac{\eta}{2}u_{\Phi}) \dd \sigma }$ for the case $u_{\Phi} < 0$. Therefore, \rfb{lbcw-cw} implies that $H_{cl}$ is lower bounded and radially unbounded.

It can be computed that
\begin{equation}\label{Hclcase4}
 \dot{H}_{cl} =\dot{V}+\dot{H}_{\circlearrowright}-\dot{y}y_{\Phi}-y\dot{y}_{\Phi}\leq - \delta \left \|L \bbm{y_{\Phi} \\x}\right \|^2.
 \end{equation}
Hence, by the radially unboundedness of $H_{cl}$, \rfb{Hclcase4} implies that $(x,y_{\Phi})$ is bounded. Using the Lasalle's invariance principle, we can conclude that the trajectory $(x,y_{\Phi})$ of \rfb{pffeedbackCW} converges to the largest invariant set contained in $M:= \{(x,y_{\Phi})\in \rline^n \times \rline|L \sbm{y_{\Phi} \\x}=0\}$.
\end{proof}

To illustrate Theorem \ref{CW-CW}, let us consider the Example \ref{ex1}, where we replace the negative feedback interconnection by a positive one.
\begin{example} \label{ex3}
\begin{equation*}
\begin{array}{rl}
\calP:& \begin{array}{rl}\dot{x} & = -3x+y_{\Phi},\\
y & = -2x+y_{\Phi},\end{array}\\
{\bf \Phi}:& \dot{y}_{\Phi}=f_1(y_{\Phi}(t),u_{\Phi}(t))\dot u_{\Phi+}(t) + f_2(y_{\Phi}(t),u_{\Phi}(t))\dot u_{\Phi-}(t)\\
&u= y_{\Phi},  \ u_{\Phi}=y,\end{array}
\end{equation*}
where $x(t)\in\R$ and the functions $f_1$, $f_2$ satisfy the hypotheses in Theorem \ref{thm2}. By using $P=\sbm{1 & -2 \\-2 & 6}$, the conditions in \rfb{LMI41} and \rfb{LMI42} hold with $L=\sbm{1&-3}$ and $\delta=2$. Also $P=P^T>\frac{1}{2} \sbm{1 & -2 \\ -2 & 4}$, i.e., $\eta=\frac{1}{2}$.
Using $H_{cl}(x,y_{\Phi})=\frac{1}{2}\sbm{y_{\Phi}& x^T}P\sbm{y_{\Phi}\\x}+H_{\circlearrowright}(y_{\Phi},-2x+y_{\Phi})-yy_{\Phi}$, routine computation shows that
\begin{equation*}
\dot H_{cl}  \leq -2(-3x+y_{\Phi})^2.
\end{equation*}

Hence, if $f_1(\gamma,v) \leq \frac{1}{4}$ and $f_2(\gamma, v)\leq \frac{1}{4}$  for all $(\gamma,v) \in \rline^2 $, then $(x,y_{\Phi})$ converges to the invariant set where $x=-\frac{1}{3}y_{\Phi}$ following Theorem \ref{CW-CW}.
\end{example} $\hfill \triangle$
\vspace{0.2cm}

\begin{figure}[h]
\centering \psfrag{P}{${\bf \Phi}$} \psfrag{C}{$\mathbf{C}$}
\psfrag{G}{$\mathbf{G}$}
 \subfigure[] {{\includegraphics[height=0.9in]{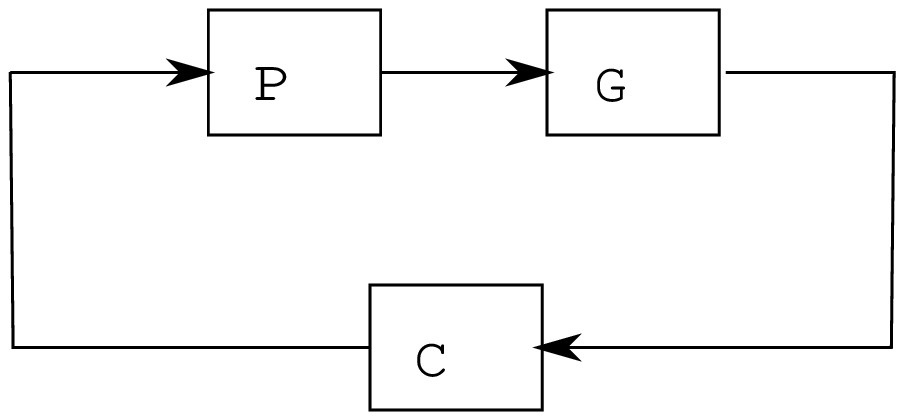}}}
 \subfigure[] {{\includegraphics[height=0.9in]{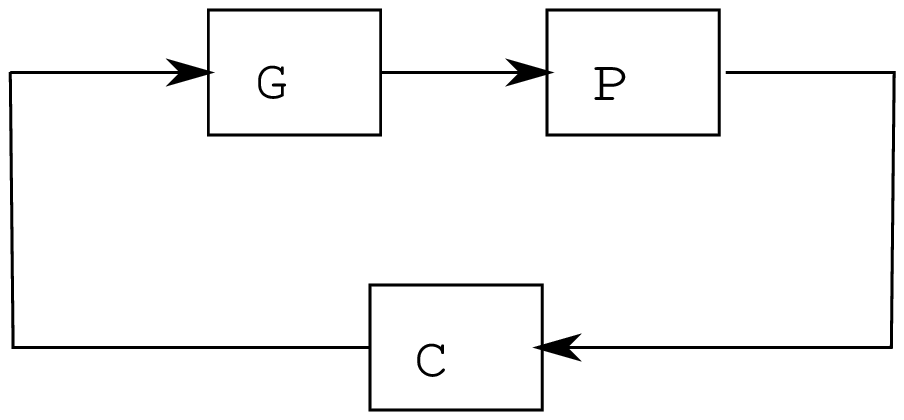}}}
\caption{Feedback interconnection of a linear plant $\mathbf{G}$, controller $\mathbf{C}$ and hysteresis operator ${\bf \Phi}$. (a) An interconnection example where the plant $\mathbf{G}$ is driven by a hysteretic actuator $\mathbf{\Phi}$; (b) An interconnection example where the dynamics of $\mathbf{G}$ is measured by a hysteretic sensor $\mathbf{\Phi}$.}
\label{controlsys} \vspace{0.2cm}
\end{figure}
\section{Controller design}\label{control-design}

The stability analysis given in the previous sections can be used to design a controller for a linear plant with hysteretic sensor/actuator. Consider the closed-loop system as shown in Figure \ref{controlsys}, where $\mathbf{G}$ and $\mathbf{C}$ are the linear plant and controller, respectively, and they are given by
\begin{equation}\label{CandG}
\mathbf{G}: \left\{\begin{array}{rl}
         \dot{x}_{G}&=A_{G}x_{G} + B_{G}u_G, \\
          y_{G}&=C_{G}x_{G}+D_{G}u_G,
        \end{array}\right. \
\mathbf{C}: \left\{\begin{array}{rl}
         \dot{x}_C&=A_Cx_C + B_Cu_C, \\
          y_C&=C_Cx_C+D_Cu_C.
        \end{array}\right.
\end{equation}
Thus depending on the location of the hysteretic element, the cascaded linear systems can be compactly written into
\begin{equation}\label{CoG}
\left.\begin{array}{rl}
         \dot{x} &= Ax+Bu \\
          y&=Cx+Du,
        \end{array}\right.
\end{equation}
where $x=\sbm{x_G\\x_C}$ and for the case of hysteretic actuator as shown in Figure \ref{controlsys}(a), $A=\sbm{A_G &0\\B_CC_G& A_C}$, $B=\sbm{B_G\\B_CD_G}$, $C=\sbm{D_CC_G & C_C}$, $D=D_CD_G$, or for the case of hysteretic sensor as shown in Figure \ref{controlsys}(b), $A=\sbm{A_G &B_GC_C\\0& A_C}$, $B=\sbm{B_GD_C\\B_C}$, $C=\sbm{C_G & D_GC_C}$, $D=D_GD_C$.
The controller design can then be carried out as follows.
\begin{itemize}
              \item {\it Control design algorithm for the case of CCW $\Phi$:}
\begin{enumerate}
  \item Determine the anhysteresis function $f_{an}$ of the Duhem operator $\Phi$ and possibly, the desired $L$.
  \item Find $\mathbf{C}$ such that either \rfb{LMI11}-\rfb{LMI3} or \rfb{LMI1}-\rfb{LMI2} holds.
  \item If \rfb{LMI11}-\rfb{LMI3} is solvable, then $\mathbf{C}$ stabilizes the closed-loop system with a negative feedback interconnection; otherwise
  \item If \rfb{LMI1}-\rfb{LMI2} is solvable, then $\mathbf{C}$ stabilizes the closed-loop system with a positive feedback interconnection.
\end{enumerate}
              \item {\it Control design algorithm for the case of CW $\Phi$:}
\begin{enumerate}
  \item Determine the functions $f_{1}$ and $f_{2}$ of the Duhem operator $\Phi$ and possibly, the desired $L$.
  \item Find $\mathbf{C}$ such that either \rfb{LMI31}-\rfb{LMI32} or \rfb{LMI41}-\rfb{LMI42} holds.
  \item If \rfb{LMI31}-\rfb{LMI32} is solvable, then $\mathbf{C}$ stabilizes the closed-loop system with a negative feedback interconnection; otherwise
  \item If \rfb{LMI41}-\rfb{LMI42} is solvable, then $\mathbf{C}$ stabilizes the closed-loop system with a positive feedback interconnection.
\end{enumerate}
\end{itemize}

Putting \rfb{CoG} into the setting of our main results in Theorem \ref{SOCW-CCW}, \ref{SOCCW-CCW}, \ref{SOCCW-CW} and \ref{CW-CW}, the invariant set is contained in $M:=\{(x_G,x_C,y_{\Phi})|N\sbm{x_G\\x_C \\y_{\Phi}}=0\}$ where the matrix $N$ can also become a design parameter for determining $\mathbf{C}$.

\section{Numerical examples}
\begin{figure}[h]
\centering \psfrag{x}{$x$} \psfrag{k}{$k$}\psfrag{PTZ}{$\Phi$}
\psfrag{b}{$b$}\psfrag{F}{ }\psfrag{m}{$m$}
  \includegraphics[height=0.9in]{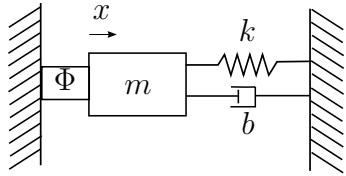}\\
  \caption{Mass-damper-spring system connected with a hysteretic actuator}\label{piezosys}
\end{figure}
As an example, we consider a mass-damper-spring system with a hysteretic actuator denoted by $\Phi$, as shown in Figure \ref{piezosys}, where $m$ is the mass, $b$ is the damping constant, $k$ is the spring constant and $x$ denotes the displacement of the mass.  Let $m=1$, $b=2$ and $k=1$, then the mass-damper-spring system is given by
\begin{align}
\dot{x}&= \left(
              \begin{array}{cc}
                0 & 1 \\
                -1 & -2 \\
              \end{array}
            \right)
x+\left(
      \begin{array}{c}
        0 \\
        1 \\
      \end{array}
    \right)
u, \nonumber \\
y&=\left(
       \begin{array}{cc}
         1 & 0 \\
       \end{array}
     \right)
x+ u. \label{mds}
\end{align}

\subsection{CCW hysteretic actuator}

Let us first consider the case when the hysteretic actuator has CCW I/O dynamics, such as piezo-actuators \cite{CHIHJER2006}. Assume that the actuator is represented by the Duhem operator \rfb{babuskamodel} where
\begin{equation}\label{exf1f2}
f_1(\gamma, v) =-\gamma+0.475v+0.3,\ f_2(\gamma, v) =\gamma-0.475v+0.3, \  \forall (\gamma, v)\in \rline^2.
\end{equation}
It can be verified that $f_{an}(v)=0.475v$ and the functions $f_1$ and $f_2$ satisfy the hypotheses given in Theorem \ref{CCW_hys}.

With $A_c=\sbm{0  & 1 \\ -2 & -4}$, $B_c=\sbm{0 \\1}$, $C_c=\sbm{-1.5& -2}$ and $D_c=1$, conditions \rfb{LMI1}-\rfb{LMI2} are solvable with $P=\sbm{1 & 1&0&-1.5 &-2\\ 1&7.74&5.51&-8.74&-15.86\\0&5.51&7.4&-5.51&-14.36\\-1.5&-8.74&-5.51&10.24&17.86\\-2&-15.86&-14.36&17.86&38.36}$ and $L=[0\ 0\ 1/4\ 0\ 0]$. Hence the controller $\mathbf{C}$ can stabilize the closed-loop system with negative feedback interconnection. In this case, $N=[0\ 1/4\ 0\ 0\ 0]$. According to Theorem \ref{SOCW-CCW}, the velocity of the mass-damper-spring system converges to zero and the position of the mass-damper-spring system converges to a constant. The closed-loop system is simulated in Matlab/Simulink with the initial condition $x(0)= [-10\ 5]^T$ and the results are shown in Figure \ref{ng_mds_sim}(a).

On the other hand, since we have $f_{an}(v)=0.475v$, then by taking $A_c=\sbm{0& 1 \\ -2& -4}$, $B_c=\sbm{0 \\1}$, $C_c=\sbm{1&1}$ and $D_c=0$, it can be checked that \rfb{LMI11}-\rfb{LMI21} holds with $\xi=0.5$ and $Q=\sbm{6& 1&-6&-2 \\ 1&4&-1&-4\\-6&-1&7&3\\-2&-4&3& 7}$. In this case $N=\sbm{1&0&-2&-3&1}$. Moreover, $f_{an}$ belongs to the sector $[0, 0.5]$. Similar to the previous case, it follows from Theorem \ref{SOCCW-CCW} that the velocity of the mass-damper-spring system converges to zero and the position of the mass-damper-spring system converges to a constant. The simulation results is shown in Figure \ref{ng_mds_sim}(b).
%
%
\begin{figure}[h]
\centering
 \subfigure[] {{\includegraphics[height=2in]{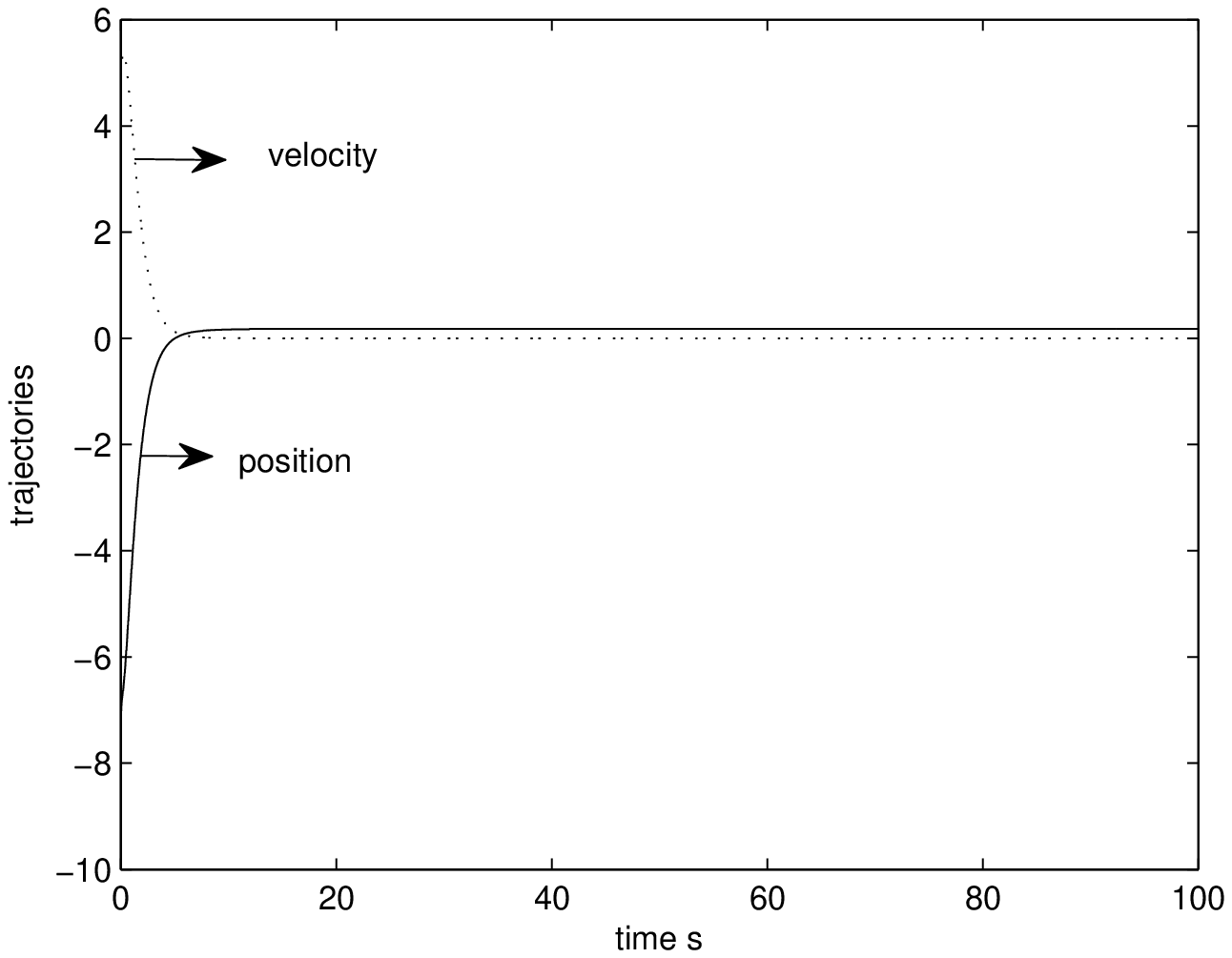}}}
 \subfigure[] {{\includegraphics[height=2in]{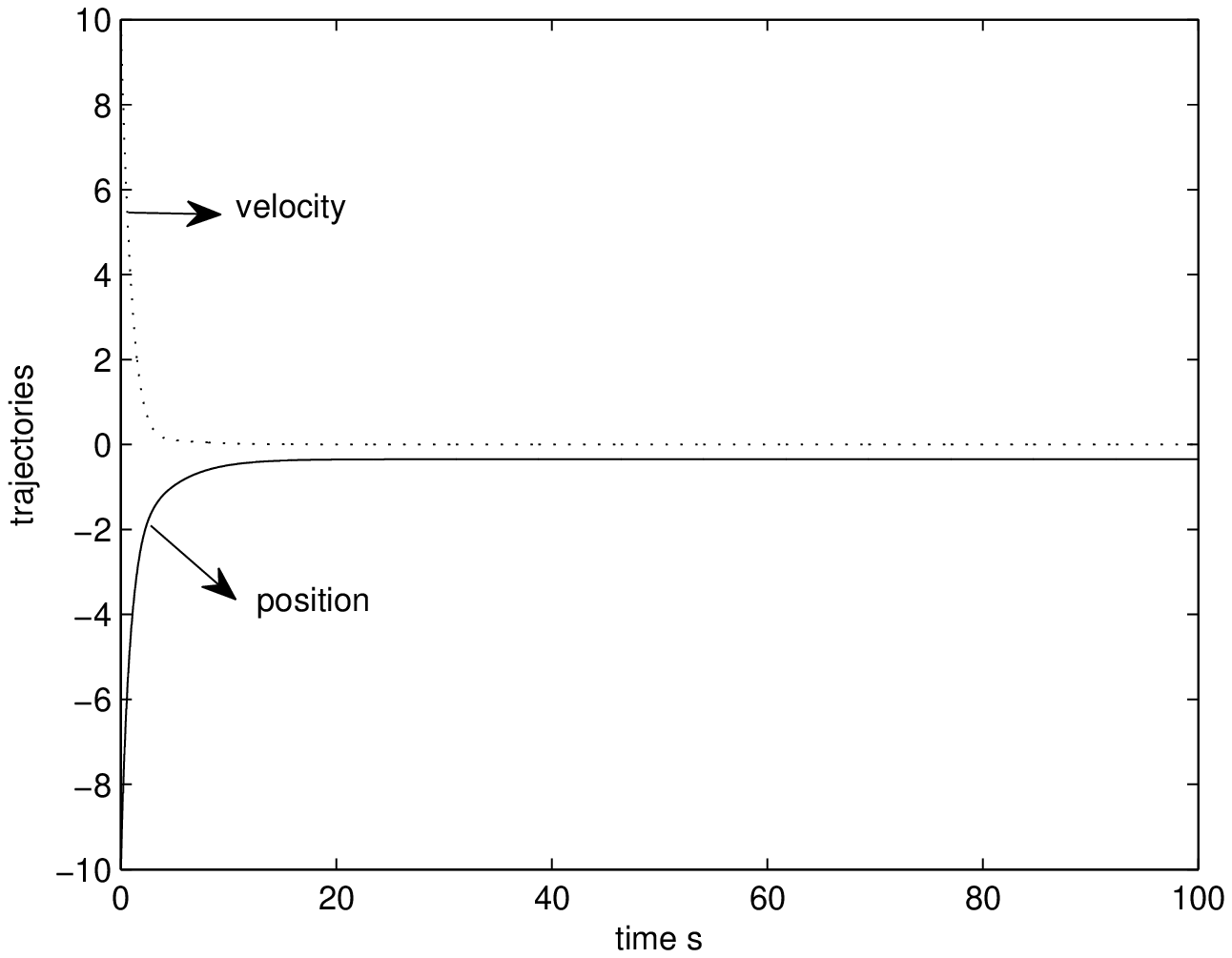}}}
\caption{Simulation results of the numerical example with CCW hysteretic actuator. (a) The negative feedback interconnection case with the initial condition $x(0)=[-10\  5]^T$; (b) The positive feedback interconnection case with the initial condition $x(0)=[-10\  10]^T$.}
\label{ng_mds_sim} \vspace{0.2cm}
\end{figure}

\subsection{CW hysteretic actuator}


For the case of a CW hysteretic actuator, see for example the magnetorheological (MR) damper used in the structure control \cite{SAKAI2003}, the mass-damper-spring system is given by \rfb{mds}. Assume that the actuator is represented by the Duhem operator \rfb{babuskamodel} where
\begin{equation}\label{exf1f2}
f_1(\gamma, v) =0.25(1-\gamma),\ f_2(\gamma, v) =0.25(1+\gamma), \  \forall (\gamma, v)\in \rline^2.
\end{equation}
The anhysteresis function for this Duhem operator is $f_{an}=0$. It can be shown that $f_1 \leq 0.25$ and $f_2 \leq 0.25$ for all $v \in \rline$. In addition $f_1$ and $f_2$ satisfy the hypotheses in Theorem \ref{thm2}, hence the Duhem operator with \rfb{exf1f2} is CW.

With $A_c=\sbm{0& 1 \\ -2& -4}$, $B_c=\sbm{0 \\1}$, $C_c=\sbm{1&1}$ and $D_c=0$, the conditions \rfb{LMI31}-\rfb{LMI32} are solvable with $P=\sbm{5& 1&-5&-2 \\ 1&3&-1&-3\\-5&-1&6&3\\-2&-3&3& 6}$. Hence the controller $\mathbf{C}$ can stabilize the closed-loop system with negative feedback interconnection. In this case, $N=[1\ 0\ -2\ -3\ 1]$. According to Theorem \ref{SOCCW-CW}, the velocity of the mass-damper-spring system converges to zero and the position of the mass-damper-spring system converges to a constant. The simulation results are shown in Figure
\ref{ng_mds_sim_cwhys}(a) with the initial condition $x(0)= [10\ 5]^T$.

Since we have $f_1 \leq 0.25$ and $f_2 \leq 0.25$ for all $v \in \rline$, by taking $A_c=\sbm{0& 1 \\ -2& -3}$, $B_c=\sbm{0 \\1}$, $C_c=\sbm{-3&-1}$ and $D_c=2$, it can be checked that \rfb{LMI41}-\rfb{LMI42} holds with $\delta=1$, $\eta=0.5$, $L=\sbm{0&1/4&0&0&0}$ and $P=\sbm{2 & 2&0&-3 &-1\\ 2&30.86&15.83&-32.86&-26.9\\0&15.83&38.26&-15.83&-51.4\\-3&-32.86&-15.83&35.86&27.9\\-1&-26.9&-51.4&27.9&74.54}$. It follows from Theorem \ref{CW-CW} that the velocity of the mass-damper-spring system converges to zero and the position of the mass-damper-spring system converges to a constant. The simulation results is shown in Figure \ref{ng_mds_sim_cwhys}(b).
%
%

\begin{figure}[h]
\centering
 \subfigure[] {{\includegraphics[height=2in]{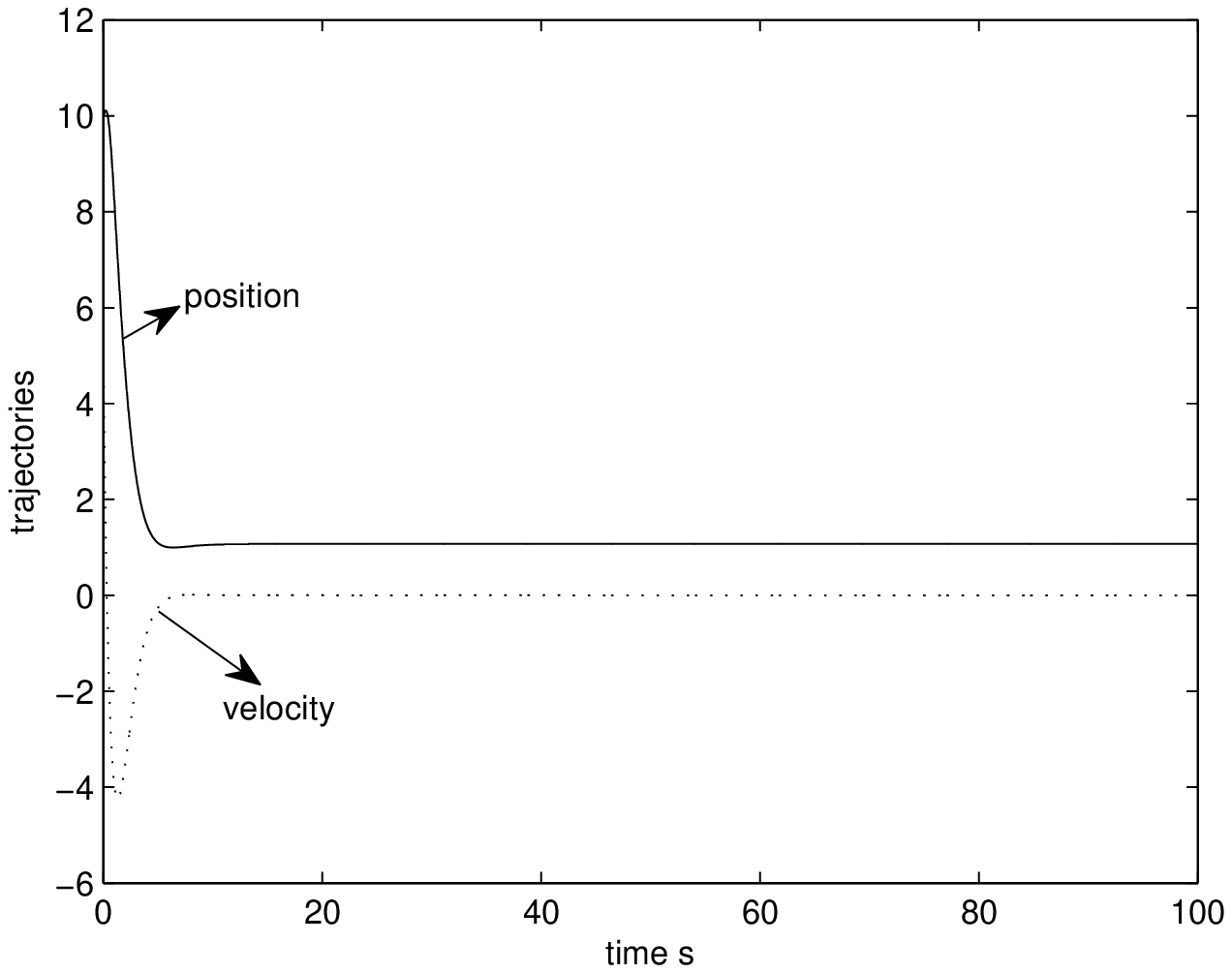}}}
 \subfigure[] {{\includegraphics[height=2in]{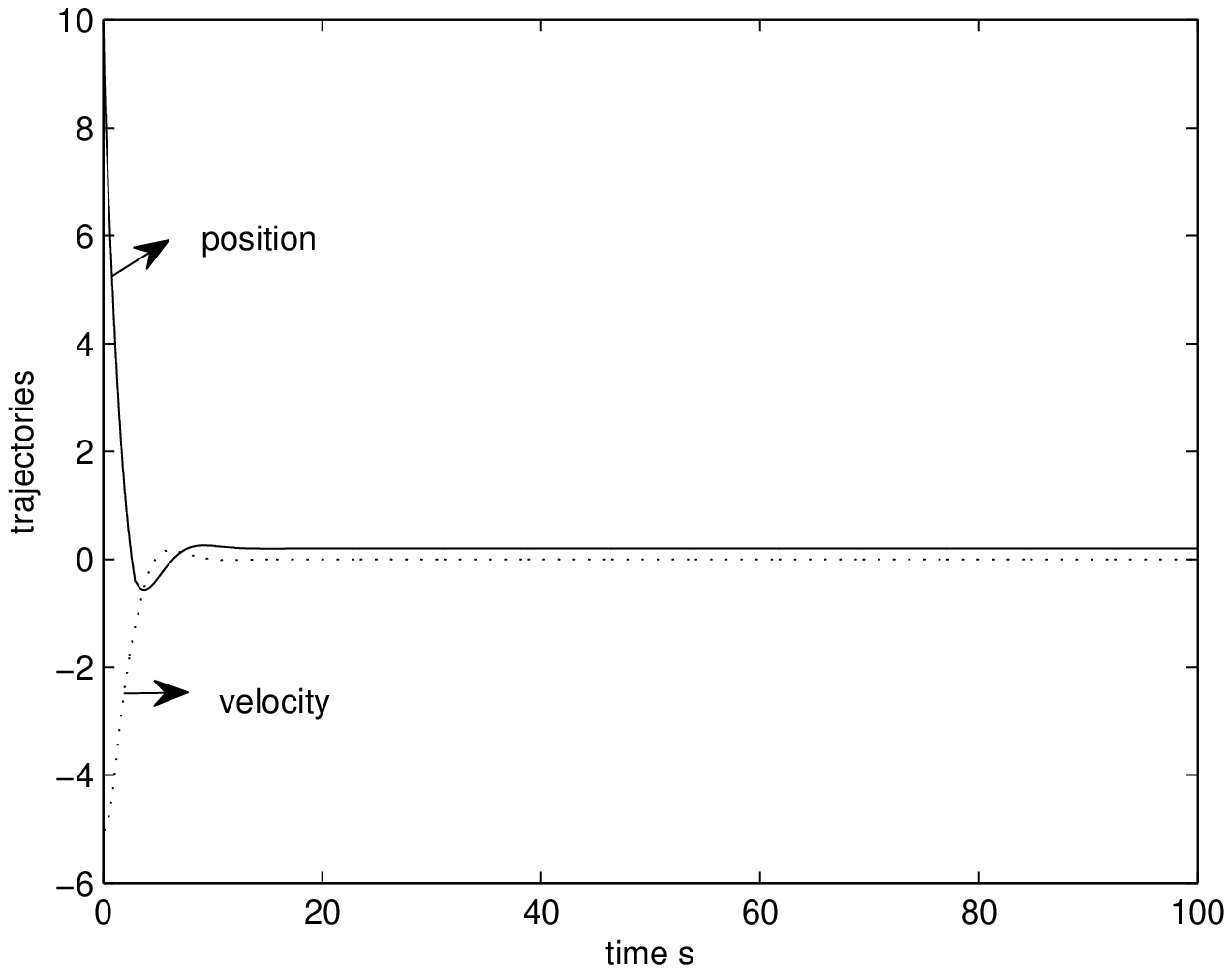}}}
\caption{Simulation results of the numerical example with CW hysteretic actuator. (a) The negative feedback interconnection case with the initial condition $x(0)=[10\  5]^T$; (b) The positive feedback interconnection case with the initial condition $x(0)=[10\  -5]^T$.}
\label{ng_mds_sim_cwhys} \vspace{0.2cm}
\end{figure}
\section{Conclusions}
It has been shown in this paper that the stability analysis of a linear system with hysteresis nonlinearity is accommodated by exploiting the I/O property of the corresponding hysteresis operator. Furthermore, the stability analysis enables a straightforward control design methodology for a plant with hysteresis nonlinearity without having to know precisely the parameters of the hysteresis operator. It offers a different paradigm in the design of controller for such systems where we do not need to define an inverse hysteresis operator which is commonly used in practice. The dissipativity approach which is used in this paper can be extended directly to nonlinear plants with hysteresis nonlinearity. One possible class of nonlinear plants which can be treated with our approach is the CCW systems as studied by Angeli \cite{ANGELI2006} and by van der Schaft \cite{SCHAFT2011}.




\end{document}